\documentclass[11pt]{amsart}
\usepackage{amsmath}
\usepackage{amssymb}
\usepackage{amsfonts}
\usepackage{amsthm}
\usepackage{enumerate}
\usepackage[mathscr]{eucal}
\usepackage{verbatim}
\usepackage{amsthm}
\usepackage{amscd}
\usepackage{appendix}

\newtheorem{theorem}{Theorem}[section]

\newtheorem{proposition}[theorem]{Proposition}

\newtheorem{corollary}[theorem]{Corollary}
\newtheorem{lemma}[theorem]{Lemma}

\theoremstyle{definition}

\newtheorem{remark}{Remark}

\newtheorem{innercustomgeneric}{\customgenericname}
\providecommand{\customgenericname}{}

\newcommand{\newcustomtheorem}[2]{%
  \newenvironment{#1}[1]
  {%
   \renewcommand\customgenericname{#2}%
   \renewcommand\theinnercustomgeneric{##1}%
   \innercustomgeneric
  }
  {\endinnercustomgeneric}
}

\newcustomtheorem{customthm}{Theorem}

\newcommand{\newcustomlemma}[2]{%
  \newenvironment{#1}[1]
  {%
   \renewcommand\customgenericname{#2}%
   \renewcommand\theinnercustomgeneric{##1}%
   \innercustomgeneric
  }
  {\endinnercustomgeneric}
}

\newcustomlemma{customlemma}{Lemma}

\newcustomlemma{customproposition}{Proposition}

\newcommand\relphantom[1]{\mathrel{\phantom{#1}}}

\numberwithin{equation}{section}

\def\MM{{\boldsymbol{M}}}
\def\a{{\alpha}}
\def\aa{{\boldsymbol{\alpha}}}

\def\xxi{\vec{\boldsymbol{\xi}}}
\def\|{{\boldsymbol{|}}}
\def\ff{{\boldsymbol{f}}}
\def\fff{\vec{\boldsymbol{f}}}
\def\vv{\vec{\boldsymbol{v}}}

\def\ll{\vec{\boldsymbol{l}}}
\def\eeta{\vec{\boldsymbol{\eta}}}

\begin{document}

\begin{thanks}
{The author is supported in part by NRF grant 2019R1F1A1044075}
\end{thanks}

\address{School of Mathematics \\
           Korea Institute for Advanced Study, Seoul\\
           Republic of Korea}
   \email{qkrqowns@kias.re.kr}

\author{Bae Jun Park}

\title[Norm equivalence and Multilinear operators]{Equivalence of (quasi-)norms on a vector-valued function space and its applications to Multilinear operators}
\keywords{}

\begin{abstract} 
In this paper we present (quasi-)norm equivalence on a vector-valued function space $L^p_A(l^q)$ and extend the equivalence to $p=\infty$ and $0<q<\infty$ in the scale of Triebel-Lizorkin space, motivated by Frazier-Jawerth \cite{Fr_Ja2}.
By applying the results, we improve the multilinear H\"ormander's multiplier theorem of Tomita \cite{Tom}, that of Grafakos-Si \cite{Gr_Si}, and the boundedness results for bilinear pseudo-differential operators, given by Koezuka-Tomita \cite{Ko_To}.
\end{abstract}

\maketitle

\section{{Introduction}}\label{introduction}

Let $T$ be multilinear operator, defined on $n$-fold products of $S(\mathbb{R}^d)$, taking values in the space of tempered distributions.
One of main problems in multilinear operator theory is $L^{p_1}\times\cdots\times L^{p_n}\to L^r$ boundedness estimates for $T$ when $1/r=1/p_1+\dots+1/p_n$, and this problem has been actively studied until recently.
For example, the multilinear Calder\'on-Zygmund theory has been  developed by Grafakos-Torres \cite{Gr_To} while particular examples in the theory have been already studied by Coifman and Meyer \cite{Co_Me1, Co_Me2, Co_Me3, Co_Me4}. 
The boundedness of bilinear Hilbert transform was obtained by Lacey and Thiele  \cite{La_Th1, La_Th2}, and
the multilinear versions of H\"ormander multiplier theorem are investigated by Tomita \cite{Tom}, Grafakos-Si \cite{Gr_Si}, Grafakos-Miyach-Tomita \cite{Gr_Mi_Tom}, Miyachi-Tomita \cite{Mi_Tom}, Grafakos-Nguyen \cite{Gr_Ng}, and Grafakos-Miyachi-Nguyen-Tomita \cite{Gr_Mi_Ng_Tom}. 
The boundedness of multilinear pseudo-differential operators was investigated by B\'enyi-Torres \cite{Be_To}, Miyachi-Tomita \cite{Mi_Tom2}, Rodr\'guez-L\'opez-Staubach \cite{Ro_St}, Michalowski-Rule-Staubach \cite{Mic_Ru_St}, Naibo \cite{Na}, and Koezuka-Tomita \cite{Ko_To}.

H\"older's inequality $L^{p_1}\cdots L^{p_n}\subset L^r$, $1/r=1/p_1+\dots+1/p_n$, is primarily required to handle such multilinear operators, but the inequality seems to be insufficient to derive $BMO$ bound when $p_j=\infty$. In turn, 
the above results mostly treat finite $p_j$'s, and occasionally extend to $L^{\infty}$ rather than $BMO$ when $p_j=\infty$.

 The aim of this paper is twofold.  The first one is to introduce (quasi-)norm equivalence on a vector-valued function space, from which $\Vert f\Vert_{BMO}$ can be expressed as $L^{\infty}(l^2)$ norm of a variant of $f$. The equivalence will enable us to still utilize H\"older's inequality to obtain some boundedness results involving $BMO$-type function spaces. The second one is to study how the equivalence can be applied to generalize previous boundedness results for multilinear operators to $BMO$-type function spaces.
 We will actually extend and improve the multilinear version of H\"ormander's multiplier theorems of Tomita \cite{Tom} and  Grafakos-Si \cite{Gr_Si}, and the boundedness result of multilinear pseudo-differential operators of Koezuka-Tomita \cite{Ko_To}.

\subsection{Equivalence of (quasi-)norms on a vector-valued function space }

 For $r>0$ let $\mathcal{E}(r)$ denote the space of all distributions whose Fourier transforms are supported in $\big\{\xi\in\mathbb{R}^d:|\xi|\leq 2r\big\}$.
Let $A>0$. For $0<p<\infty$ and $0<q\leq \infty$ or for $p=q=\infty$ we define
 \begin{equation*}
L_A^p(l^q):=\big\{\{f_k\}_{k\in\mathbb{Z}}\subset S':f_k\in\mathcal{E}(A2^k), \big\Vert \{ f_k\}_{k\in\mathbb{Z}}\big\Vert_{L^p(l^q)}<\infty \big\}.
\end{equation*}
 Then it is known in \cite{Tr} that $L_A^p(l^q)$ is a quasi-Banach space (Banach space if $p,q\geq 1$) with a (quasi-)norm $\Vert \cdot\Vert_{L^p(l^q)}$. 
 We will study some (quasi-)norm equivalence on $L_A^p(l^q)$  and one of main results is an extension of the norm equivalence to the case $p=\infty$ and $0<q<\infty$ in the scale of Triebel-Lizorkin space.
 
Let $\mathcal{D}$ denote the set of all dyadic cubes in $\mathbb{R}^d$, and for each $k\in\mathbb{Z}$ let  $\mathcal{D}_{k}$ be the subset of  $\mathcal{D}$ consisting of the cubes with side length $2^{-k}$.
For $k\in \mathbb{Z}$, $\sigma>0$, and $0<t\leq \infty$ let
\begin{equation}\label{generalpeetre}
\mathfrak{M}_{\sigma,2^k}^{t}f(x):=2^{kd/t}\Big\Vert \frac{f_k(x-\cdot)}{(1+2^k|\cdot|)^{\sigma}}\Big\Vert_{L^t},
\end{equation}
which is a generalization of the Peetre's maximal function $\mathfrak{M}_{\sigma,2^k}f(x):=\mathfrak{M}_{\sigma,2^k}^{\infty}f(x)$. 
We refer to Section \ref{preliminary} for properties of the operator $\mathfrak{M}_{\sigma,2^k}^t$.

\begin{theorem}\label{mainequiv}
Let $0<q<\infty$, $\sigma>d/t>d/q$, $0<\gamma<1$, and $\mu\in\mathbb{Z}$. Suppose $A>0$ and $f_k\in \mathcal{E}(A2^k)$ for each $k\in\mathbb{Z}$. For $Q\in\mathcal{D}$ there exists a proper measurable subset $S_Q$ of $Q$, depending on $\gamma,q,\sigma,t,\{f_k\}_{k\in\mathbb{Z}}$, such that $|S_Q|>(1-\gamma)|Q|$ and
\begin{align}\label{equivkeyclaim}
&\sup_{P\in \mathcal{D},l(P)\leq 2^{-\mu}}{\Big(\frac{1}{|P|}\int_P{\sum_{k=-\log_2{l(P)}}^{\infty}{|f_k(x)|^q}}dx \Big)^{1/q}} \nonumber
\\
&\relphantom{=} \approx \Big\Vert \Big\{  \sum_{Q\in\mathcal{D}_k}{\Big( \inf_{y\in Q}{\mathfrak{M}_{\sigma,2^k}^tf_k(y)}\Big)\chi_{S_Q}}   \Big\}_{k\geq \mu} \Big\Vert_{L^{\infty}(l^q)}, \quad \text{uniformly in }~\mu.
\end{align}
\end{theorem}
We note that the constant in (\ref{equivkeyclaim}) is independent of $\{f_k\}_{k\in\mathbb{Z}}$, just depending on $\gamma$. 
The equivalence in Theorem \ref{mainequiv} can be compared with the estimate in Lemma \ref{character2} that for $0<p<\infty$ or $p=q=\infty$ 
 \begin{equation}\label{briefcharacter}
\big\Vert \{f_k\}_{k\in\mathbb{Z}}\big\Vert_{L^p(l^q)}\approx \Big\Vert \Big\{ \sum_{Q\in\mathcal{D}_k}{\Big(\inf_{y\in Q}{ \mathfrak{M}_{\sigma,2^k}^tf_k(y)}\Big)\chi_{S_Q }}\Big\}_{k\in\mathbb{Z}}\Big\Vert_{L^p(l^q)}, \quad \{f_k\}_{k\in\mathbb{Z}}\in L_A^p(l^q)
\end{equation}
if  $\sigma>d/t>d/\min{(p,q)}$.
Note that for $1<p<\infty$, according to Littlewood-Paley theory,
\begin{equation*}
\Vert f\Vert_{L^p}\approx \big\Vert \big\{ \phi_k\ast f\big\}_{k\in\mathbb{Z}}\big\Vert_{L^p(l^2)}
\end{equation*} and, using (\ref{briefcharacter}),
this is also comparable to
\begin{equation}\label{lpcharacter}
\Big\Vert \Big\{ \sum_{Q\in\mathcal{D}_k}{\Big(\inf_{y\in Q}{ \mathfrak{M}_{\sigma,2^k}^t\big(\phi_k\ast f\big)(y)}\Big)\chi_{S_Q}}\Big\}_{k\in\mathbb{Z}}\Big\Vert_{L^p(l^2)}
\end{equation}
where $\{\phi_k\}_{k\in\mathbb{Z}}$ is a homogeneous Littlewood-Paley partition of unity, defined in Section \ref{preliminary}.
On the other hand, using a deep connection between $BMO$ and Carleson measure,
\begin{equation*}
\Vert f\Vert_{BMO}\approx \sup_{P\in\mathcal{D}}{\Big( \frac{1}{|P|}\int_P{\sum_{k=-\log_2{l(P)}}^{\infty}{\big|\phi_k\ast f(x)\big|^2}}dx\Big)^{1/2}}.
\end{equation*}
The main value of Theorem \ref{mainequiv} is that $\Vert \cdot \Vert_{BMO}$ can be expressed in the form $\Vert \cdot \Vert_{L^{\infty}(l^2)}$ as an extension of (\ref{lpcharacter}) to $p=\infty$.
\begin{corollary}\label{maincor}
Let  $\sigma>d/t>d/2$ and $0<\gamma<1$.  For $Q\in\mathcal{D}$ there exists a proper measurable subset $S_Q$ of $Q$, depending on $\gamma,\sigma,t,f$, such that
$|S_Q|>(1-\gamma) |Q|$ and
\begin{equation*}
\Vert f\Vert_{BMO}\approx \Big\Vert \Big\{ \sum_{Q\in\mathcal{D}_k}{\Big( \inf_{y\in Q}{\mathfrak{M}_{\sigma,2^k}^t\big(\phi_k\ast f\big)(y) }\Big)\chi_{S_Q}}\Big\}_{k\in\mathbb{Z}}\Big\Vert_{L^{\infty}(l^2)}.
\end{equation*}
\end{corollary}

A simple application of Corollary \ref{maincor} is the inequality
\begin{equation}\label{bmoh1}
\big|\langle f,g\rangle \big|\lesssim \Vert f\Vert_{BMO}\Vert g\Vert_{H^1}.
\end{equation}
This provides one direction of the duality between $H^1$ and $BMO$, which was first announced in \cite{Fe} and proved in \cite{Car, Fe_St}.
It can be also proved in a different way, using Corollary \ref{maincor} and H\"older's inequality. The proof will be given in Appendix \ref{appendixa}.

\subsection{H\"ormander multiplier theorem for multilinear operators}

For simplicity we use the notation $\xxi:=(\xi_1,\dots,\xi_n)$.
For $m\in L^{\infty}\big( (\mathbb{R}^d)^n\big)$ the $n$-linear multiplier operator $T_{m}$ is defined by
\begin{equation*}
T_{m}\big(f_1,\dots,f_n \big)(x):={\int_{(\mathbb{R}^d)^n}{m(\xxi)\Big(\prod_{j=1}^{n}\widehat{f_j}(\xi_j)\Big)e^{2\pi i\langle x,\sum_{j=1}^{n}{\xi_j} \rangle}}d\xxi}
\end{equation*} for $f_j\in S(\mathbb{R}^d)$.
 Let $\vartheta^{(n)}\in S((\mathbb{R}^d)^n)$ have the properties that $0\leq \vartheta^{(n)}\leq 1$, $\vartheta^{(n)}=1$ for $2^{-1}\leq |\xxi| \leq 2$, and $Supp(\vartheta^{(n)})\subset \big\{\xxi\in (\mathbb{R}^d)^n: 2^{-2}\leq |\xxi|\leq 2^2 \big\}$.
 Define
\begin{equation*}
\mathcal{L}_s^{r,\vartheta^{(n)}}[m]:=\sup_{l\in\mathbb{Z} }{\big\Vert m(2^l\cdot_{1},\dots,2^l\cdot_{n})\vartheta^{(n)}\big\Vert_{L_s^r((\mathbb{R}^d)^n)}}.
\end{equation*}
We recall the multilinear multiplier theorem of Tomita \cite{Tom}.
\begin{customthm}{A}\label{previousmulti0}
Suppose  $1<p,p_1,\dots,p_n<\infty$ and $1/p=1/p_1+\cdots +1/p_n$. 
If $m\in L^{\infty}((\mathbb{R}^d)^n )$ satisfies $\mathcal{L}_s^{2,\vartheta^{(n)}}[m]<\infty$ for  $s>nd/2$, then there exists a constant $C>0$ so that
\begin{equation*}
\big\Vert T_{m}\big( f_1,\dots,f_n\big)\big\Vert_{L^p}\leq C\mathcal{L}_s^{2,\vartheta^{(n)}}[m] \prod_{j=1}^{n}\Vert f_j\Vert_{L^{p_j}}.
\end{equation*}
\end{customthm}

Another boundedness result was obtained by Grafakos-Si \cite{Gr_Si}
\begin{customthm}{B}\label{previousmulti}
Let $0<p<\infty$ and  $1/p=1/p_1+\cdots+1/p_n$.
Suppose $1<r\leq 2$ and $m$ satisfies $\mathcal{L}_s^{r,\vartheta^{(n)}}[m]<\infty$ for $s>nd/r$.
Then there exists a number $\delta>0$, satisfying $0<\delta\leq r-1$, such that 
\begin{equation*}
\big\Vert T_{m}\big( f_1,\dots,f_n\big)\big\Vert_{L^p}\lesssim \mathcal{L}_s^{r,\vartheta^{(n)}}[m] \prod_{j=1}^{n}\Vert f_j\Vert_{L^{p_j}}.
\end{equation*} whenever $r-\delta<p_j<\infty$ for $1\leq j\leq n$.

\end{customthm}
Note that Theorem \ref{previousmulti} takes into account a broader range of $p$ by giving stronger assumptions on $s$, while, under the same assumption $s>nd/2$ (when $r=2$), the estimate in Theorem \ref{previousmulti} is a partial result of Theorem \ref{previousmulti0}. We also refer to \cite{Fu_Tom1, Fu_Tom2, Gr_Mi_Tom, Gr_Mi_Ng_Tom, Mi_Tom} for further results.

We will generalize Theorem \ref{previousmulti0} and \ref{previousmulti}.
Let 
\begin{equation*}
X^p:=\Big\{\begin{array}{ll}
H^p\quad & \text{if}\quad p<\infty\\
BMO \quad  & \text{if}\quad p=\infty 
\end{array}.
\end{equation*}

\begin{theorem}\label{multipliermain0}
Let $1<p<\infty$ and $1<p_{i,j}\leq \infty$, $1\leq i,j\leq n$, satisfy
\begin{equation}\label{conditionp}
\frac{1}{p}=\frac{1}{p_{i,1}}+\dots+\frac{1}{p_{i,n}} \quad \text{for  }~ 1\leq i\leq n. 
\end{equation}
Suppose  $m$ satisfies
$\mathcal{L}_s^{2,\vartheta^{(n)}}[m]<\infty$ for $s>nd/2$.
Then
\begin{equation*}
\big\Vert T_m\big( f_1,\dots,f_n\big)\big\Vert_{L^p}\lesssim \mathcal{L}_s^{2,\vartheta^{(n)}}[m]\sum_{i=1}^{n}{\Big(\Vert f_i\Vert_{X^{p_{i,i}}} \prod_{1\leq j\leq n,j\not= i}{\Vert f_j\Vert_{L^{p_{i,j}}}}\Big)}.
\end{equation*}
\end{theorem}

\begin{theorem}\label{multipliermain}
Let $0<p<\infty$ and $0<p_{i,j}\leq \infty$, $1\leq i,j\leq n$, satisfy (\ref{conditionp}).
 Suppose $1<u\leq 2$, $0<r\leq 2$, and $m$ satisfies
$\mathcal{L}_s^{u,\vartheta^{(n)}}[m]<\infty$ for $s>nd/r$.
Then there exists a number $\delta>0$, satisfying $0<\delta\leq r$,  such that
\begin{equation*}
\big\Vert T_m\big( f_1,\dots,f_n\big)\big\Vert_{L^p}\lesssim \mathcal{L}_s^{u,\vartheta^{(n)}}[m]\sum_{i=1}^{n}{\Big(\Vert f_i\Vert_{X^{p_{i,i}}} \prod_{1\leq j\leq n,j\not= i}{\Vert f_j\Vert_{H^{p_{i,j}}}}\Big)}
\end{equation*}
whenever $r-\delta<p_{i,j}\leq \infty$ for $1\leq i,j\leq n$.

\end{theorem}
We remark that, under the same hypothesis $s>nd/r$,
the condition $\mathcal{L}_s^{r,\vartheta^{(n)}}[m]<\infty$ in Theorem \ref{previousmulti} is improved to  $\mathcal{L}_s^{u,\vartheta^{(n)}}[m]<\infty$ for any $1<u\leq 2$ in Theorem \ref{multipliermain}.
Due to the independence of $r$ in $\mathcal{L}_s^{u,\vartheta^{(n)}}[m]<\infty$, one has better freedom in the range $0<r\leq 2$ and $r-\delta<p_{i,j}\leq \infty$.

\subsection{Multilinear pseudo-differential operators of type $(1,1)$}

The $n$-linear H\"ormander symbol class $\MM_n\mathcal{S}_{1,1}^{m}$ consists of all $a\in C^{\infty}\big((\mathbb{R}^d)^{n+1}\big)$ having the property that for all multi-indices $\alpha_1$,$\dots$,$\alpha_n$,$\beta$ there exists a constant $C=C_{\aa,\beta}$ such that
\begin{equation*}
\big| \partial_{\xxi}^{\aa}\partial_{x}^{\beta}a(x,\xxi)\big|\leq C\Big(1+\sum_{j=1}^{n}|\xi_j| \Big)^{m-|\aa|+|\beta|}
\end{equation*} where $\aa:=(\alpha_1,\dots,\alpha_n)$ and $|\aa|:=|\alpha_1|+\dots +|\alpha_n|$.
The corresponding $n$-linear pseudo-differential operator $T_{[a]}$ is defined by 
\begin{equation*}
T_{[a]}\big( f_1,\dots,f_n\big)(x):=\int_{(\mathbb{R}^d)^n}{a(x,\xxi)\prod_{j=1}^{n}\widehat{f_j}(\xi_j)e^{2\pi i\langle x,\sum_{j=1}^{n}\xi_j\rangle}}d\xxi
\end{equation*} for $f_1,\dots,f_n\in S(\mathbb{R}^d)$.
Denote by  $Op\MM_n\mathcal{S}_{1,1}^m$ the class of $n$-linear pseudo-differential operators with symbols in $\MM_n\mathcal{S}_{1,1}^{m}$.
Bilinear pseudo-differential operators(n=2) in $Op\MM_2\mathcal{S}_{1,1}^{0}$ have  bilinear Calder\'on-Zygmund kernels, but in general they are not bilinear Calder\'on-Zygmund operators. In particular, they do not always give rise to a mapping $L^{p_1}\times L^{p_2}\to L^p$ for $1<p,p_1,p_2\leq \infty$ with $1/p=1/p_1+1/p_2$.

The boundedness properties of operators in $Op\MM_2\mathcal{S}_{1,1}^{0}$ have been studied by B\'enyi-Torres \cite{Be_To}, and B\'enyi-Nahmod-Torres \cite{Be_Na_To} in the scale of Lebesgue-Sobolev spaces. To be specific, B\'enyi-Torres \cite{Be_To} proved that if $a\in\MM_2\mathcal{S}_{1,1}^0$, then
\begin{equation*}
\big\Vert T_{[a]}(f_1,f_2)\big\Vert_{L_s^p}\lesssim \Vert f_1\Vert_{L_s^{p_1}}\Vert f_2\Vert_{L^{p_2}}+\Vert f_1\Vert_{L^{p_1}}\Vert f_2\Vert_{L_s^{p_2}}
\end{equation*} for $1<p_1,p_2,p<\infty$, $1/p_1+1/p_2=1/p$, and $s>0$.
Moreover, this result was generalized to $a\in\MM_2\mathcal{S}_{1,1}^m$, $m\in\mathbb{R}$, by B\'enyi-Nahmod-Torres \cite{Be_Na_To}.
Naibo \cite{Na} investigated bilinear pseudo-differential operators on Triebel-Lizorkin spaces and  Koezuka-Tomita \cite{Ko_To} slightly developed the result of Naibo.
These works can be readily extended to multilinear operators.
For $a\in \MM_n\mathcal{S}_{1,1}^{m}$ and $N\in\mathbb{N}_0$ we define
\begin{equation*}
\Vert a\Vert_{\MM_n\mathcal{S}_{1,1,N}^{m}}:=\max{\Big[\sup\Big(1+\sum_{j=1}^{n}|\xi_j|\Big)^{-m+|\aa|-|\beta|}\big|\partial_{\xxi}^{\aa}\partial_{x}^{\beta}a(x,\xxi)\big| \Big]}
\end{equation*} where the supremum is taken over $(x,\xxi)\in (\mathbb{R}^d)^{n+1}$ and the maximum is taken over $|\a_1|,\dots, |\a_n|,|\beta|\leq N$.
For $0<p,q\leq \infty$ let 
\begin{equation*}\label{taucondition}
\tau_{p}:=d/\min{(1,p)}-d, \quad \tau_{p,q}:=d/\min{(1,p,q)}-d.
\end{equation*}
\begin{customthm}{C}\cite{Ko_To, Na}\label{theorembb}
Let $0<p<\infty$, $0<q\leq \infty$, $m\in\mathbb{R}$, and $a\in\MM_n\mathcal{S}_{1,1}^{m}$. Let   $0<p_{i,j}<\infty$, $1\leq i,j\leq n$, satisfy (\ref{conditionp}).
If  \begin{equation}\label{weakscondition}
s>\Big\{\begin{array}{ll}
\tau_{p,q}\quad & \text{if}\quad q<\infty\\
\tau_{p,\infty}+d\quad  & \text{if}\quad q=\infty 
\end{array},
\end{equation} then there exists a positive integer $N$ such that
\begin{equation*}
\big\Vert T_{[a]}\big( f_1,\dots,f_n\big) \big\Vert_{F_p^{s,q}}\lesssim  \Vert a\Vert_{\MM_n\mathcal{S}_{1,1,N}^{m}}\sum_{i=1}^{n}{\Big( \Vert f_i\Vert_{F_{p_{i,i}}^{s+m,q}}\prod_{1\leq j\leq n, j\not= i}{\Vert f_j\Vert_{h^{p_{i,j}}}}\Big)}
\end{equation*} for $f_1,\dots,f_n \in S(\mathbb{R}^d)$. Moreover, the inequality also holds for $p_{i,j}=\infty$, $i\not=j$.
\end{customthm}
We refer the reader to Section \ref{preliminary} for notations and definitions of some function spaces.
Recall that $h^p=L^p$ for $1<p\leq  \infty$ and $F_p^{s,2}=h^p_s$ for $0<p<\infty$.

Note that the condition (\ref{weakscondition}) in Theorem \ref{theorembb} is due to the multiplier theorem of Triebel \cite{Tr}. Recently, the author \cite{Park5} has improved the result of Triebel, sharpening the condition on  $s$, and extending the multiplier theorem to $p=\infty$ in the scale of Triebel-Lizorkin space. Using this result and Theorem \ref{mainequiv} we will extend Theorem \ref{theorembb} to the full range $0<p,p_{i,j}\leq \infty$ with the weaker condition $s>\tau_{p,q}$, instead of (\ref{weakscondition}).

\begin{theorem}\label{pseudotheorem}
Suppose $0<p,q\leq \infty$, $m\in\mathbb{R}$, and $a\in \MM_n\mathcal{S}_{1,1}^{m}$. 
 Let $0<p_{i,j}\leq \infty$, $1\leq i,j\leq n$, satisfy (\ref{conditionp}).
If $s>\tau_{p,q}$, then there exists a positive integer $N$ such that 
\begin{equation*}
\big\Vert T_{[a]}\big(f_1,\dots,f_n\big) \big\Vert_{F_p^{s,q}}\lesssim  \Vert a\Vert_{\MM_n\mathcal{S}_{1,1,N}^{m}}\sum_{i=1}^{n}{\Big( \Vert f_i\Vert_{F_{p_{i,i}}^{s+m,q}}\prod_{1\leq j\leq n, j\not= i}{\Vert f_j\Vert_{h^{p_{i,j}}}}\Big)}
\end{equation*}
for $f_1,\dots,f_n \in S(\mathbb{R}^d)$.
\end{theorem}

As a corollary, from $h_s^p=F_{p}^{s,2}$ and $bmo_s=F_{\infty}^{s,2}$, the following estimates hold.
Let \begin{equation*}
Y_s^p:=\Big\{\begin{array}{ll}
h_s^p\quad & \text{if}\quad p<\infty\\
bmo_s \quad  & \text{if}\quad p=\infty 
\end{array}.
\end{equation*}
\begin{corollary}\label{maincorollary}
Suppose $0<p\leq \infty$, $m\in\mathbb{R}$, and $a\in \MM_n\mathcal{S}_{1,1}^{m}$. 
 Let $\{p_{i,j}\}_{1\leq i,j\leq n}$ satisfy $0<p_{i,j}\leq \infty$ and \eqref{conditionp}.
If $s>\tau_{p}$, then there exist positive integers $N>0$ such that
\begin{equation*}
\big\Vert T_{[a]}\big(f_1,\dots,f_n \big) \big\Vert_{Y_s^p}\lesssim  \Vert a\Vert_{\MM_n\mathcal{S}_{1,1,N}^{m}}\sum_{i=1}^{n}{\Big( \Vert f_i\Vert_{Y^{p_{i,i}}_{s+m}}\prod_{1\leq j\leq n, j\not= i}{\Vert f_j\Vert_{h^{p_{i,j}}}}\Big)}
\end{equation*}
for $f_1,\dots,f_n \in S(\mathbb{R}^d)$.

\end{corollary}

\subsubsection*{Generalization of Kato-Ponce inequality}

The classical Kato-Ponce commutator estimate \cite{Ka_Po} plays a key role in the wellposedness theory of Navier-Stokes and Euler equations in Sobolev spaces.
The commutator estimate has been recast later on into the following fractional Leibniz rule, so called Kato-Ponce inequality.
Let $J^s:=(1-\Delta)^{s/2}$ be the (inhomogeneous) fractional Laplacian operator.
Then
\begin{equation}\label{Kato}
\big\Vert J^s(fg)\big\Vert_{L^p}\lesssim \big\Vert J^sf\big\Vert_{L^{p_1}}\Vert g\Vert_{L^{p_2}}+\Vert f\Vert_{L^{\widetilde{p}_1}}\big\Vert J^sg\big\Vert_{L^{\widetilde{p}_2}}
\end{equation} where $1/p=1/p_1+1/p_2=1/\widetilde{p}_1+1/\widetilde{p}_2$, $1<p<\infty$, and $1<p_1,p_2,\widetilde{p}_1,\widetilde{p}_2\leq \infty$.
Grafakos-Oh \cite{Gr_Oh} and Muscalu-Schlag \cite{Mu_Sc} extended the inequality (\ref{Kato}) to the wider range $1/2<p<\infty$ under the assumption that $s>\tau_p$ or $s\in 2\mathbb{N}$.
Recently, Naibo-Thomson \cite{Na_Tho} extend it to (weighted) local Hardy space for $0<p,p_1,p_2,\widetilde{p}_1,\widetilde{p}_2<\infty$. 

\begin{customthm}{D}\label{previousKato}
Let $0<p,p_1,p_2,\widetilde{p}_1,\widetilde{p}_2< \infty$  satisfy $1/p=1/p_1+1/p_2=1/\widetilde{p}_1+1/\widetilde{p}_2$. Suppose $s>\tau_p$. Then for $f,g\in S(\mathbb{R}^d)$ one has
\begin{equation*}
\big\Vert J^s(fg)\big\Vert_{h^p}\lesssim \Vert J^sf\Vert_{h^{p_1}}\Vert g\Vert_{h^{p_2}}+\Vert f\Vert_{h^{\widetilde{p}_1}}\Vert J^sg\Vert_{h^{\widetilde{p}_2}}
\end{equation*}
\end{customthm}

Additionally, the case $p=\infty$ was settled by Bourgain-Li \cite{Bo_Li} and $BMO$ estimates for homogeneous Laplacian operators $D^s:=(-\Delta)^{s/2}$ was established by Brummer-Naibo \cite{Br_Na}.

As a consequence of Corollary \ref{maincorollary} in the case $a\equiv 1$, one obtains the following extension of Kato-Ponce inequality, which includes an endpoint case of $bmo$ type.
\begin{corollary}\label{KatoPonce}
Let $0<p,p_1,p_2,\widetilde{p}_1,\widetilde{p}_2\leq \infty$  satisfy $1/p=1/p_1+1/p_2=1/\widetilde{p}_1+1/\widetilde{p}_2$.
 Suppose $s>\tau_p$. Then for $f,g \in S(\mathbb{R}^d)$ one has
\begin{equation*}
\big\Vert J^s(fg)\big\Vert_{Y^p}\lesssim  \Vert J^sf\Vert_{Y^{p_1}}\Vert g\Vert_{h^{p_2}}+\Vert f\Vert_{h^{\widetilde{p}_1}}\Vert J^sg\Vert_{Y^{\widetilde{p}_2}}
\end{equation*}
where 
 \begin{equation*}
Y^p:=\Big\{\begin{array}{ll}
h^p\quad & \text{if}\quad p<\infty\\
bmo \quad  & \text{if}\quad p=\infty 
\end{array}.
\end{equation*}
\end{corollary}

\hfill

The main ingredient in the proof of Theorem \ref{mainequiv} is the maximal inequalities for $\mathfrak{M}_{\sigma,2^k}^{t}$, which are stated in Lemma \ref{maximal2}.
Then for $0<q<\infty$ one obtains that for any proper measurable subset $S_Q$ of $Q$ the left hand side of (\ref{equivkeyclaim}) is comparable to
\begin{equation*}
 \sup_{P\in \mathcal{D},l(P)\leq 2^{-\mu}}{\Big(\frac{1}{|P|}\int_P{\sum_{k=-\log_2{l(P)}}^{\infty}{\Big|  \sum_{Q\in\mathcal{D}_k}{\Big( \inf_{y\in Q}{\mathfrak{M}_{\sigma,2^k}^tf_k(y)}\Big)\chi_{S_Q}(x)}   \Big|^q}}dx \Big)^{1/q}}.
\end{equation*}
Thus, in order to prove Theorem \ref{mainequiv} one needs to show that there exists a subset $S_Q$ of $Q$ such that
\begin{equation*}
\sup_{P\in\mathcal{D}, l(P)\leq 2^{-\mu}}\Big(\frac{1}{|P|}\int_P{\sum_{k=-\log_2{l(P)}}^{\infty}\big| \cdots \big|^q}dx \Big)^{1/q}\approx \Big\Vert \Big(\sum_{k=\mu}^{\infty}{\big| \cdots\big|^q} \Big)^{1/q}\Big\Vert_{L^{\infty}}
\end{equation*}
and one direction is clear because the essential supremum of a function dominates the supremum of averages.  For the other direction, we take advantage of "$\gamma$-median" and its nice properties.
The proofs of Theorem \ref{multipliermain0}, \ref{multipliermain}, and \ref{pseudotheorem} are based on the Littlewood-Paley decomposition, breaking down operator $T$ in the form
$T=\sum_{i=1}^{n}{T_i}$. Then we establish 
\begin{equation*}
\Vert T_i(f_1,\dots,f_n)\Vert_{X}\lesssim \Vert f_i\Vert_{Y_i}\prod_{1\leq j\leq n, j\not=i}{\Vert f_j\Vert_{Z_j}}
\end{equation*} where $X$, $Y_i$'s, and $Z_j$'s are suitable spaces which appear in Theorem  \ref{multipliermain0}, \ref{multipliermain}, and \ref{pseudotheorem}.
The improvement of the condition $\mathcal{L}_s^{u,\vartheta^{(n)}}[m]<\infty$ in Theorem \ref{multipliermain} is provided by using Nikolskii's inequality. 
Theorem \ref{multipliermain} (with $r=u=2$) implies the case $2\leq p_1,\dots,p_n\leq \infty$ in Theorem \ref{multipliermain0}, and thus we first present the proof of Theorem \ref{multipliermain}.
Then the technique of transposes for multilinear operators in \cite{Tom} completes the proof of Theorem \ref{multipliermain0}.

The paper is organized as follows. Some preliminary results are given in Section \ref{preliminary}.
In Section \ref{charactersection} we discuss several (quasi-)norm equivalence and prove Theorem \ref{mainequiv}. In Section \ref{multimultipliersection} we prove Theorem \ref{multipliermain0} and \ref{multipliermain}.
 Section \ref{pseudosection} is devoted to the proof of Theorem \ref{pseudotheorem}.

We make some convention on notation.
Let $\mathbb{N}$ and $\mathbb{Z}$ be the collections of all natural numbers and all integers, respectively, and $\mathbb{N}_0:=\mathbb{N}\cup \{0\}$. 
We will use the symbol $A\lesssim B$ to indicate that $A\leq CB$ for some constant $C>0$, possibly different at each occurrence, and $A\approx B$ if $A\lesssim B$ and $B\lesssim A$ simultaneously.

\section{Preliminaries}\label{preliminary}

\subsection{Function spaces}\label{functionspaces}

Let $\Phi_0\in S$ satisfy $Supp(\widehat{\Phi_0})\subset \big\{\xi\in\mathbb{R}^d:|\xi|\leq 1 \big\}$ and $\widehat{\Phi_0}(\xi)=1$ for $|\xi|\leq 1/2$. Define $\phi:=\Phi_0-2^{-d}\Phi(2^{-1}\cdot)$ and $\phi_k:=2^{kd}\phi(2^k\cdot)$.
Then $\{\Phi_0\}\cup \{\phi_k\}_{k\in\mathbb{N}}$ and $\{\phi_k\}_{k\in\mathbb{Z}}$ form inhomogeneous and homogeneous Littlewood-Paley partition of unity, respectively.
Note that $Supp(\widehat{\phi_k})\subset \big\{\xi\in\mathbb{R}^d: 2^{k-2}\leq |\xi|\leq 2^{k}\big\}$ and
\begin{equation*}
\widehat{\Phi_0}(\xi)+\sum_{k\in\mathbb{N}}{\widehat{\phi_k}(\xi)}=1, \quad \text{(inhomogeneous)}
\end{equation*}
\begin{equation*}
\sum_{k\in\mathbb{Z}}{\widehat{\phi_k}(\xi)}=1, \quad \xi\not= 0.  \quad \text{(homogeneous)}
\end{equation*}

For $0<p,q\leq \infty$ and $s\in\mathbb{R}$, inhomogeneous Triebel-Lizorkin space $F_p^{s,q}$ is the collection of all $f\in S'$ such that
\begin{equation*}
\Vert f\Vert_{F_p^{s,q}}:=\Vert \Phi_0\ast f\Vert_{L^p}+\big\Vert \big\{2^{sk}\phi_k\ast f \big\}_{k\in\mathbb{N}} \big\Vert_{L^p(l^q)}<\infty, \quad 0<p<\infty ~\text{ or }~ p=q=\infty,
\end{equation*}
\begin{equation*}
\Vert f\Vert_{F_{\infty}^{s,q}}:=\Vert \Phi_0\ast f\Vert_{L^{\infty}}+\sup_{P\in\mathcal{D}, l(P)<1}{\Big( \frac{1}{|P|}\int_P{\sum_{k=-\log_2{l(P)}}^{\infty}{2^{skq}\big| \phi_k\ast f(x)\big|^q}}dx\Big)^{1/q}}, \quad 0<q<\infty
\end{equation*} where the supremum is taken over all dyadic cubes whose side length is less than $1$.
Similarly, homogeneous  Triebel-Lizorkin space $\dot{F}_p^{s,q}$ is defined to be the collection of all $f\in S'/\mathcal{P}$ (tempered distribution modulo polynomials) such that 
 \begin{equation*}
 \Vert f\Vert_{\dot{F}_p^{s,q}}:=\big\Vert \big\{ 2^{sk}\phi_k\ast f\big\}_{k\in\mathbb{Z}}\big\Vert_{L^p(l^q)}<\infty, \quad 0<p<\infty ~\text{ or }~ p=q=\infty,
 \end{equation*}
 \begin{equation*}
\Vert f\Vert_{\dot{F}_{\infty}^{s,q}}:=\sup_{P\in\mathcal{D}}{\Big( \frac{1}{|P|}\int_P{\sum_{k=-\log_2{l(P)}}^{\infty}{2^{skq}\big| \phi_k\ast f(x)\big|^q}}dx\Big)^{1/q}}, \quad 0<q<\infty.
\end{equation*} 
 
 Then these spaces provide a general framework that unifies classical function spaces.
\begin{align*}
& L^p\text{space} &\dot{{F}}_p^{0,2} = {F}_p^{0,2}=L^p & &1<p<\infty \\
&\text{Hardy space} &\dot{{F}}_p^{0,2}={H}^p , \quad   F_p^{0,2}=h^p & &0<p\leq 1\\
&\text{Fractional Sobolev space}&\dot{{F}}_p^{{s},2}=\dot{L}^p_{{s}}, \quad  {F}_p^{{s},2}=L^p_{{s}}  & &  1<p<\infty\\
&\text{Hardy-Sobolev space}&\dot{{F}}_p^{{s},2}={H}^p_{{s}}, \quad  {F}_p^{{s},2}=h^p_{{s}}  & &  0<p\leq 1\\
&BMO, bmo &\dot{{F}}_{{\infty}}^{0,2}=BMO,  \quad  {{F}}_{{\infty}}^{0,2}=bmo\\
&\text{Sobolev-$BMO$} &\dot{{F}}_{{\infty}}^{s,2}=BMO_s,  \quad  {{F}}_{{\infty}}^{s,2}=bmo_s.
\end{align*}

Recall that for $0<p\leq \infty$
\begin{equation}\label{localhardy}
 \Vert f\Vert_{h^p}\approx\Big\Vert \sup_{k\in\mathbb{N}_0}{\big|\Phi_k\ast f\big|}\Big\Vert_{L^p}, \quad \Vert f\Vert_{H^p}\approx\Big\Vert \sup_{k\in\mathbb{Z}}{\big|\Phi_k\ast f\big|}\Big\Vert_{L^p}
\end{equation}
where  $\Phi_k:= 2^{kd}\Phi_0(2^k\cdot)$, and the space $bmo$ is a localized version of $BMO$ defined as the set of locally integrable functions $f$ satisfying
\begin{equation*}
\Vert f\Vert_{bmo}:=\sup_{l(Q)\leq 1}{\frac{1}{|Q|}\int_Q{\big| f(x)-f_Q\big|}dx}+\sup_{l(Q)>1}{\frac{1}{|Q|}\int_Q{|f(x)|}dx}<\infty
\end{equation*} where $f_Q$ is the average of $f$ over a cube $Q$.
Moreover, for $s\in\mathbb{R}$
\begin{equation*}
\Vert f\Vert_{H_s^p}:=\Vert D^sf\Vert_{H^p} \qquad \Vert f\Vert_{BMO_s}:=\Vert D^sf\Vert_{BMO}
\end{equation*} 
\begin{equation*}
\Vert f\Vert_{h_s^p}:=\Vert J^sf\Vert_{h^p} \qquad  \Vert f\Vert_{bmo_s}:=\Vert J^sf\Vert_{bmo}
\end{equation*} 
where $\widehat{D^sf}(\xi):=|\xi|^s\widehat{f}(\xi)$ and $\widehat{J^sf}(\xi):=(1+|\xi|^2)^{s/2}\widehat{f}(\xi)$ are the fractional Laplacian operators as before. 
It is known that $(H^1)^*=BMO$, $(h^1)^*=bmo$, and 
$h^p=H^p=L^p$ for $1<p\leq \infty$.
See \cite{Fr_Ja2, Go, Tr} for more details.

\subsection{Maximal inequalities}
Let $\mathcal{M}$ be the Hardy-Littlewood maximal operator, defined by
\begin{equation*}
\mathcal{M}f(x):=\sup_{x\in Q}{\frac{1}{|Q|}\int_Q{|f(y)|}dy}
\end{equation*} where the supremum is taken over all cubes containing $x$, and for $0<r<\infty$ let $\mathcal{M}_rf:=\big(\mathcal{M}(|f|^r)\big)^{1/r}$. 
Then Fefferman-Stein's vector-valued maximal inequality in \cite{Fe_St} says that for $0<p<\infty$, $0<q\leq \infty$, and $0<r<\min{(p,q)}$ one has
\begin{equation}\label{hlmax}
\big\Vert \big\{\mathcal{M}_rf_k\big\}_{k\in\mathbb{Z}}\big\Vert_{L^p(l^q)}\lesssim \big\Vert \{ f_k\}_{k\in\mathbb{Z}}\big\Vert_{L^p(l^q)}.
\end{equation}
Clearly, (\ref{hlmax}) also holds when $p=q=\infty$.

We now introduce a variant of Hardy-Littlewood maximal function.
For $\epsilon\geq0$, $r>0$, and $k\in\mathbb{Z}$, let $\mathcal{M}_r^{k,\epsilon}$be defined by
\begin{align*}
\mathcal{M}_r^{k,\epsilon}f(x)&:=\sup_{x\in Q, 2^kl(Q)\leq 1}{\Big( \frac{1}{|Q|}\int_{Q}{|f(y)|^r}dy  \Big)^{1/r}}\\
  &\relphantom{=}+\sup_{x\in Q, 2^{k}l(Q)>1}{\big(2^kl(Q)\big)^{-\epsilon}\Big( \frac{1}{|Q|}\int_Q{|f(y)|^r}dy  \Big)^{1/r}}. \nonumber
\end{align*}
Note that $\mathcal{M}_r^{k,\epsilon}f(x)$ is decreasing function of $\epsilon$, and   $\mathcal{M}_r^{k,0}f(x)\approx \mathcal{M}_rf(x)$.
Then the following maximal inequality holds for the case $p=\infty$ and $0<q<\infty$.
\begin{lemma}\label{maximal1}\cite{Park1}
Let $0<r<q<\infty$, $\epsilon>0$, and $\mu\in\mathbb{Z}$. Suppose $A>0$ and $f_k\in\mathcal{E}(A2^k)$ for each $k\in\mathbb{Z}$.
Then one has
\begin{equation*}
\sup_{P\in\mathcal{D}_{\mu}}{\Big(  \frac{1}{|P|}\int_P{  \sum_{k=\mu}^{\infty}{  \big( \mathcal{M}_r^{k,\epsilon}f_k(x)  \big)^{q}      }    }dx  \Big)^{1/q}} \lesssim \sup_{R\in\mathcal{D}_{\mu}}{\Big(  \frac{1}{|R|}\int_R{  \sum_{k=\mu}^{\infty}{   |f_k(x)|^q   }    }dx  \Big)^{1/q}}. 
\end{equation*}
Here, the implicit constant of the inequality is independent of $\mu$.
\end{lemma}

We now continue with some properties of the operator $\mathfrak{M}_{\sigma,2^k}^{t}$, defined in (\ref{generalpeetre}).
\begin{lemma}\label{comparelemma}
Let $\sigma>0$, $0<t\leq s\leq \infty$, and $k\in\mathbb{Z}$.
 Suppose $A>0$ and $f\in\mathcal{E}(A2^k)$.
Then
\begin{equation*}
\mathfrak{M}_{\sigma,2^k}^{s}f(x)\lesssim \mathfrak{M}_{\sigma,2^k}^{t}f(x).
\end{equation*}

\end{lemma}

\begin{lemma}\label{mcomposition}
Let $\sigma>0$, $0<t\leq s\leq \infty$, and $k\in \mathbb{Z}$. 
Then
\begin{equation*}
\mathfrak{M}_{\sigma,2^k}^{s}\mathfrak{M}_{\sigma,2^k}^tf(x)\lesssim \mathfrak{M}_{\sigma,2^k}^tf(x).
\end{equation*}

\end{lemma}

\begin{lemma}\label{comparelemma2}
Let  $0<t<\infty$,   $\sigma>d/t$, $0<\epsilon<\sigma -d/t$, and $k\in \mathbb{Z}$. Suppose $A>0$ and $f\in\mathcal{E}(A2^k)$. Then 
\begin{equation*}
\mathfrak{M}_{\sigma,2^k}^{t}f(x)\lesssim \mathcal{M}_t^{k,\epsilon}f(x).
\end{equation*}

\end{lemma}
The proofs of Lemma \ref{comparelemma}, \ref{mcomposition}, and \ref{comparelemma2} will be given in Appendix \ref{appendixb}.

Elementary considerations reveal that for $\sigma>0$ and $Q\in\mathcal{D}_k$
\begin{equation}\label{infmax}
 \sup_{y\in Q}{|f(y)|}\lesssim \inf_{y\in Q}{\mathfrak{M}_{\sigma,2^k}f(y)}
\end{equation}
and then it follows from Lemma \ref{mcomposition} that for $0<t\leq \infty$
\begin{equation}\label{infmax2}
\sup_{y\in Q}{\mathfrak{M}_{\sigma,2^k}^tf(y)}\lesssim \inf_{y\in Q}{\mathfrak{M}_{\sigma,2^k}\mathfrak{M}_{\sigma,2^k}^{t}f(y)}\lesssim \inf_{y\in Q}{\mathfrak{M}_{\sigma,2^k}^tf(y)}
\end{equation}  if $f\in \mathcal{E}(A2^k)$ for some $A>0$.

In addition, Lemma \ref{comparelemma2}, (\ref{hlmax}), and Lemma \ref{maximal1} lead immediately to the following maximal inequalities.
\begin{lemma}\label{maximal2}
Let $0<p,q\leq \infty$ and $\sigma>d/t> d/\min{(p,q)}$. Suppose $A>0$ and $f_k\in\mathcal{E}(A2^k)$ for each $k\in\mathbb{Z}$. 
\begin{enumerate}
\item For $0<p<\infty$ or $p=q=\infty$
\begin{equation*}
 \big\Vert \big\{ \mathfrak{M}_{\sigma,2^k}^tf_k\big\}_{k\in\mathbb{Z}}\big\Vert_{L^p(l^q)}\lesssim \big\Vert \{f_k\}_{k\in\mathbb{Z}}\big\Vert_{L^p(l^q)}.
\end{equation*}
\item For $p=\infty$, $0<q<\infty$, and $\mu\in\mathbb{Z}$
\begin{align*}
&\sup_{P\in\mathcal{D}_{\mu}}{\Big( \frac{1}{|P|}\int_P{\sum_{k=-\log_2{l(P)}}^{\infty}{\big(\mathfrak{M}_{\sigma,2^k}^tf_k(x) \big)^q}}dx\Big)^{1/q}}\\
&\lesssim \sup_{P\in\mathcal{D}_{\mu}}{\Big( \frac{1}{|P|}\int_P{\sum_{k=-\log_2{l(P)}}^{\infty}{|f_k(x) |^q}}dx\Big)^{1/q}}
\end{align*} where the constant in the inequality is independent of $\mu$.

\end{enumerate}

\end{lemma}

If $\min{(p,q)}\leq t\leq \infty$ and $\sigma>d/\min{(p,q)}$ then we may choose $0<t_0<\min{(p,q)}$ so that $\sigma>d/t_0>d/\min{(p,q)}$ and 
\begin{equation*}
\mathfrak{M}_{\sigma,2^k}^tf_k(x)\lesssim \mathfrak{M}_{\sigma,2^k}^{t_0}f_k(x)
\end{equation*}
for $f_k\in \mathcal{E}(A2^k)$. Then  as a consequence of Lemma \ref{maximal2} one obtains the following lemma.
\begin{lemma}\label{maximal3}
Let $0<p,q\leq \infty$, $\min{(p,q)}\leq t\leq \infty$, and $\sigma> d/\min{(p,q)}$. Suppose $A>0$ and $f_k\in\mathcal{E}(A2^k)$ for each $k\in\mathbb{Z}$. 
\begin{enumerate}
\item For $0<p<\infty$ or $p=q=\infty$
\begin{equation*}
 \big\Vert \big\{ \mathfrak{M}_{\sigma,2^k}^tf_k\big\}_{k\in\mathbb{Z}}\big\Vert_{L^p(l^q)}\lesssim \big\Vert \{f_k\}_{k\in\mathbb{Z}}\big\Vert_{L^p(l^q)}.
\end{equation*}
\item For $p=\infty$, $0<q<\infty$, and $\mu\in\mathbb{Z}$
\begin{align*}
&\sup_{P\in\mathcal{D}_{\mu}}{\Big( \frac{1}{|P|}\int_P{\sum_{k=-\log_2{l(P)}}^{\infty}{\big(\mathfrak{M}_{\sigma,2^k}^tf_k(x) \big)^q}}dx\Big)^{1/q}}\\
&\lesssim \sup_{P\in\mathcal{D}_{\mu}}{\Big( \frac{1}{|P|}\int_P{\sum_{k=-\log_2{l(P)}}^{\infty}{|f_k(x) |^q}}dx\Big)^{1/q}}
\end{align*} where the constant in the inequality is independent of $\mu$.

\end{enumerate}

\end{lemma}

\subsection{Multiplier theorem for $L^p_A(l^q)$}

The next lemma states a vector-valued version of H\"ormander's multiplier theorem. It was partially proved by Triebel \cite[1.6.3, 2.4.9]{Tr} and was completed by the author \cite{Park5} recently.
\begin{customlemma}{E}\label{main}
Let $0<p,q\leq \infty$ and $\mu\in\mathbb{Z}$.
Suppose $f_k\in\mathcal{E}(A2^k)$ for each $k\in\mathbb{Z}$, and $\{m_k\}_{k\in\mathbb{Z}}$ satisfies
\begin{equation*}
\sup_{l\in\mathbb{Z}}{\big\Vert m_l(2^{l}\cdot)\big\Vert_{L_s^2}}<\infty \qquad \text{for }~ s>d/\min{(1,p,q)}-d/2.
\end{equation*}
\begin{enumerate}
\item For $0<p<\infty$ or $p=q=\infty$, 
\begin{equation*}
\big\Vert \big\{ \big( m_k \widehat{f_{k}}\big)^{\vee}\big\}_{k\in\mathbb{Z}}\big\Vert_{L^p(l^q)}\lesssim \sup_{l\in\mathbb{Z}}{\big\Vert m_l(2^{l}\cdot)\big\Vert_{L_s^2}} \big\Vert \big\{ f_k\big\}_{k\in\mathbb{Z}}\big\Vert_{L^p(l^q)}.
\end{equation*}
\item For $p=\infty$ and $0<q<\infty$
\begin{align*}
&\sup_{P\in\mathcal{D}, l(P)\leq 2^{-\mu}}{\Big(\frac{1}{|P|}\int_P{\sum_{k=-\log_2{l(P)}}^{\infty}{\big| \big(m_k\widehat{f_{k}} \big)^{\vee}(x)\big|^q}}dx \Big)^{1/q}}\\
&\lesssim  \sup_{l\geq \mu}{\big\Vert m_l(2^{l}\cdot)\big\Vert_{L_s^2}}\sup_{P\in\mathcal{D}, l(P)\leq 2^{-\mu}}{\Big(\frac{1}{|P|}\int_P{\sum_{k=-\log_2{l(P)}}^{\infty}{|f_{k}(x)|^q}}dx \Big)^{1/q}}
\end{align*} uniformly in $\mu$.
\end{enumerate}
\end{customlemma}

Observe that for $j\geq 0$
\begin{equation}\label{mkobserve}
\big( m_k\widehat{f_{k+j}}\big)^{\vee}(x)=\big(m_k(2^j\cdot)\big( f_{k+j}(2^{-j}\cdot)\big)^{\wedge} \big)^{\vee}(2^jx)
\end{equation}
and the following result can be verified with the use of a change of variables.
\begin{lemma}\label{propo}
Let $0<p,q\leq \infty$, $j\geq 0$, and $\mu\in\mathbb{Z}$.
Suppose $f_k\in\mathcal{E}(A2^k)$ for each $k\in\mathbb{Z}$, and $\{m_k\}_{k\in\mathbb{Z}}$ satisfies
\begin{equation*}
\sup_{l\in\mathbb{Z}}{\big\Vert m_l(2^{l+j}\cdot)\big\Vert_{L_s^2}}<\infty \qquad \text{for }~ s>d/\min{(1,p,q)}-d/2.
\end{equation*}
\begin{enumerate}
\item For $0<p<\infty$ or $p=q=\infty$, 
\begin{equation*}
\big\Vert \big\{ \big( m_k \widehat{f_{k+j}}\big)^{\vee}\big\}_{k\in\mathbb{Z}}\big\Vert_{L^p(l^q)}\lesssim \sup_{l\in\mathbb{Z}}{\big\Vert m_l(2^{l+j}\cdot)\big\Vert_{L_s^2}} \big\Vert \{f_k\}_{k\in\mathbb{Z}}\big\Vert_{L^p(l^q)}
\end{equation*}
uniformly in $j$.
\item For $p=\infty$ and $0<q<\infty$
\begin{align*}
&\sup_{P\in\mathcal{D}, l(P)\leq 2^{-\mu}}{\Big(\frac{1}{|P|}\int_P{\sum_{k=-\log_2{l(P)}}^{\infty}{\big| \big(m_k\widehat{f_{k+j}} \big)^{\vee}(x)\big|^q}}dx \Big)^{1/q}}\\
&\lesssim  \sup_{l\geq \mu}{\big\Vert m_l(2^{l+j}\cdot)\big\Vert_{L_s^2}}\sup_{P\in\mathcal{D}, l(P)\leq 2^{-\mu}}{\Big(\frac{1}{|P|}\int_P{\sum_{k=-\log_2{l(P)}}^{\infty}{|f_{k+j}(x)|^q}}dx \Big)^{1/q}}
\end{align*} uniformly in $\mu$ and $j$.
\end{enumerate}
\end{lemma}

\begin{proof}

(1) By using (\ref{mkobserve}) and Lemma \ref{main} (1), one has
\begin{align*}
\big\Vert \big\{ \big( m_k \widehat{f_{k+j}}\big)^{\vee}\big\}_{k\in\mathbb{Z}}\big\Vert_{L^p(l^q)}&=2^{-jd/p}\big\Vert \big\{ \big( m_k(2^j\cdot)\big(f_{k+j}(2^{-j}\cdot) \big)^{\wedge}\big)^{\vee} \big\}_{k\in\mathbb{Z}}  \big\Vert_{L^p(l^q)}\\
&\lesssim \sup_{l\in\mathbb{Z}}{\big\Vert m_l(2^{l+j}\cdot)\big\Vert_{L_s^2}}2^{-jd/p}\big\Vert \big\{f_{k+j}(2^{-j}\cdot) \big\}_{k\in\mathbb{Z}}\big\Vert_{L^p(l^q)}\\
&=\sup_{l\in\mathbb{Z}}{\big\Vert m_l(2^{l+j}\cdot)\big\Vert_{L_s^2}} \big\Vert \{f_k\}_{k\in\mathbb{Z}}\big\Vert_{L^p(l^q)}
\end{align*} uniformly in $j$,  since $f_{k+j}(2^j\cdot)\in\mathcal{E}(A2^k)$.

(2) Similarly, (\ref{mkobserve}) and Lemma \ref{main} (2) yield that
\begin{align*}
&\sup_{P\in\mathcal{D}, l(P)\leq 2^{-\mu}}{\Big(\frac{1}{|P|}\int_P{\sum_{k=-\log_2{l(P)}}^{\infty}{\big| \big(m_k\widehat{f_{k+j}} \big)^{\vee}(x)\big|^q}}dx \Big)^{1/q}}\\
&=\sup_{R\in\mathcal{D},l(R)\leq 2^{-\mu+j}}{\Big(\frac{1}{|R|}\int_R{\sum_{k=-\log_2{l(R)}+j}^{\infty}{\big|    \big( m_k(2^j\cdot)\big(f_{k+j}(2^{-j}\cdot) \big)^{\wedge}\big)^{\vee}(x)      \big|^q}}dx \Big)^{1/q}}\\
&\lesssim \sup_{l\geq \mu}{\big\Vert m_l(2^{l+j})\big\Vert_{L_s^2}}\sup_{R\in\mathcal{D},l(R)\leq 2^{-\mu+j}}{\Big( \frac{1}{|R|}\int_R{\sum_{k=-\log_2{l(R)}}^{\infty}{\big|f_{k+j}(2^{-j}x) \big|^q}}dx\Big)^{1/q}}\\
&=\sup_{l\geq \mu}{\big\Vert m_l(2^{l+j}\cdot)\big\Vert_{L_s^2}}\sup_{P\in\mathcal{D}, l(P)\leq 2^{-\mu}}{\Big(\frac{1}{|P|}\int_P{\sum_{k=-\log_2{l(P)}}^{\infty}{|f_{k+j}(x)|^q}}dx \Big)^{1/q}}
\end{align*} uniformly in $j$.

\end{proof}

\section{Equivalence of (quasi-)norms by using $\mathfrak{M}^t_{\sigma,2^k}$ }\label{charactersection}

Let $f_k\in\mathcal{E}(A2^k)$ for some $A>0$ and $\ff:=\{f_k\}_{k\in\mathbb{Z}}$.
For convenience in notation we will occasionally write
\begin{equation*}
\mathfrak{M}_{\sigma}(\ff):=\big\{\mathfrak{M}_{\sigma,2^k}f_k \big\}_{k\in\mathbb{Z}}, \quad \mathfrak{M}_{\sigma}^t(\ff):=\big\{\mathfrak{M}_{\sigma,2^k}^tf_k \big\}_{k\in\mathbb{Z}}.
\end{equation*}
It follows from (\ref{infmax}) and Lemma \ref{comparelemma} that for $\sigma>0$ and $0<t\leq \infty$
\begin{equation*}
|f_k(x)|=\sum_{Q\in\mathcal{D}_k}{|f_k(x)|\chi_Q(x)}\lesssim \sum_{Q\in\mathcal{D}_k}{\Big( \inf_{y\in Q}{\mathfrak{M}_{\sigma,2^k}^tf_k(y)}\Big)\chi_Q(x)}\leq \mathfrak{M}_{\sigma,2^k}^tf_k(x).
\end{equation*}
Then Lemma \ref{maximal2} (1) gives the (quasi-)norm equivalence
\begin{equation}\label{character1}
\Vert \ff\Vert_{L^p(l^q)}\approx \Big\Vert \Big\{ \sum_{Q\in\mathcal{D}_k}{\Big( \inf_{y\in Q}{\mathfrak{M}_{\sigma,2^k}^tf_k(y)}\Big)\chi_Q}\Big\}_{k\in\mathbb{Z}}\Big\Vert_{L^p(l^q)}\approx \big\Vert \mathfrak{M}_{\sigma}^t\ff\big\Vert_{L^p(l^q)}
\end{equation} for  $\sigma>d/t>d/\min{(p,q)}$ if $0<p<\infty$ or $p=q=\infty$.
Similarly, if $0<q<\infty$ and $\mu\in\mathbb{Z}$, then
\begin{align}
&\sup_{P\in \mathcal{D},l(P)\leq 2^{-\mu}}{\Big(\frac{1}{|P|}\int_P{\sum_{k=-\log_2{l(P)}}^{\infty}{|f_k(x)|^q}}dx \Big)^{1/q}}\nonumber\\
&\approx \sup_{P\in \mathcal{D},l(P)\leq 2^{-\mu}}{\Big(\frac{1}{|P|}\int_P{\sum_{k=-\log_2{l(P)}}^{\infty}{\Big|  \sum_{Q\in\mathcal{D}_k}{\Big( \inf_{y\in Q}{\mathfrak{M}_{\sigma,2^k}^tf_k(y)}\Big)\chi_Q(x)}   \Big|^q}}dx \Big)^{1/q}} \label{qequiv2}\\
&=\sup_{P\in\mathcal{D}, l(P)\leq 2^{-\mu}}{\Big( \frac{1}{|P|}\sum_{k=-\log_2{l(P)}}^{\infty}{\sum_{Q\in\mathcal{D}_k,Q\subset P}{\Big(\inf_{y\in Q}{\mathfrak{M}_{\sigma,2^k}^tf_k(y)} \Big)^q|Q|}}\Big)^{1/q}} \label{qequiv}
\end{align} for  $\sigma>d/t>d/q$.

\begin{lemma}\label{character2}
Let $0<p,q\leq\infty$, $\sigma>d/t>d/\min{(p,q)}$, $0<\gamma<1$, and $\mu\in\mathbb{Z}$. 
For each $Q\in\mathcal{D}$ let $S_Q$ be a measurable subset of $Q$ with $|S_Q|>(1-\gamma) |Q|$.
Suppose $A>0$ and $f_k\in \mathcal{E}(A2^k)$ for each $k\in\mathbb{Z}$.
\begin{enumerate}
\item For $0<p<\infty$ or $p=q=\infty$
\begin{equation*}
\Vert \mathbf{f}\Vert_{L^p(l^q)}\approx_{\gamma} \Big\Vert \Big\{ \sum_{Q\in\mathcal{D}_k}{\Big( \inf_{y\in Q}{\mathfrak{M}_{\sigma,2^k}^tf_k(y)}\Big)\chi_{S_Q}}\Big\}_{k\in\mathbb{Z}}\Big\Vert_{L^p(l^q)}.
\end{equation*}
\item For $p=\infty$ and $0<q<\infty$
\begin{align*}
&\sup_{P\in \mathcal{D},l(P)\leq 2^{-\mu}}{\Big(\frac{1}{|P|}\int_P{\sum_{k=-\log_2{l(P)}}^{\infty}{|f_k(x)|^q}}dx \Big)^{1/q}}\\
&\approx_{\gamma}  \sup_{P\in \mathcal{D},l(P)\leq 2^{-\mu}}{\Big(\frac{1}{|P|}\int_P{\sum_{k=-\log_2{l(P)}}^{\infty}{\Big|  \sum_{Q\in\mathcal{D}_k}{\Big( \inf_{y\in Q}{\mathfrak{M}_{\sigma,2^k}^tf_k(y)}\Big)\chi_{S_Q}(x)}   \Big|^q}}dx \Big)^{1/q}}.
\end{align*}
\end{enumerate}
\end{lemma}
Note that the constants in the estimates are independent of $S_Q$ as long as $|S_Q|>(1-\gamma) |Q|$.
\begin{proof}[Proof of Lemma \ref{character2}]
The second assertion follows immediately from (\ref{qequiv}) and the condition $|S_Q|>(1-\gamma) |Q|$.
Thus we only pursue the first one. Assume $0<p<\infty$ or $p=q=\infty$.
  Since $\chi_Q\geq \chi_{S_Q}$ one direction is obvious due to (\ref{character1}).
We will base the converse on the pointwise estimate that for $0<r<\infty$ 
\begin{equation}\label{inversechi}
\chi_{Q}(x)\lesssim \mathcal{M}_r(\chi_{S_Q})(x)\chi_Q(x),
\end{equation}
which is due to the observation that for $x\in Q$
\begin{equation*}
1<\frac{1}{(1-\gamma)^{1/r}}\frac{|S_Q|^{1/r}}{|Q|^{1/r}}=\frac{1}{(1-\gamma)^{1/r}}\Big(\frac{1}{|Q|}\int_Q{\chi_{S_Q}(y)}dy \Big)^{1/r}\leq (1-\gamma)^{-1/r}\mathcal{M}_r(\chi_{S_Q})(x).
\end{equation*} 

Choose $r<p,q$ and then apply (\ref{character1}) and (\ref{inversechi}) to obtain
\begin{align*}
\Vert \mathbf{f}\Vert_{L^p(l^q)}&\lesssim  \Big\Vert \Big\{ \sum_{Q\in\mathcal{D}_k}{\Big( \inf_{y\in Q}{\mathfrak{M}_{\sigma,2^k}^tf_k(y)\Big)\mathcal{M}_r(\chi_{S_Q})\chi_Q}}\Big\}_{k\in\mathbb{Z}}\Big\Vert_{L^p(l^q)}\\
&\leq \Big\Vert \Big\{ \Big(\inf_{y\in Q}{\mathfrak{M}_{\sigma,l(Q)^{-1}}^tf_{-\log_2{l(Q)}}(y)}\Big)\mathcal{M}_r(\chi_{S_Q})\Big\}_{Q\in\mathcal{D}}\Big\Vert_{L^p(l^q)} \\
&\lesssim  \Big\Vert \Big\{  \Big( \inf_{y\in Q}{\mathfrak{M}_{\sigma,l(Q)^{-1}}^tf_{-\log_2{l(Q)}}(y)}\Big) \chi_{S_Q}\Big\}_{Q\in\mathcal{D}}\Big\Vert_{L^p(l^q)}\\
&= \Big\Vert \Big\{ \sum_{Q\in\mathcal{D}_k}{\Big( \inf_{y\in Q}{\mathfrak{M}_{\sigma,2^k}^tf_k(y)}\Big) \chi_{S_Q}}\Big\}_{k\in\mathbb{Z}}\Big\Vert_{L^p(l^q)}
\end{align*} where  the maximal inequality (\ref{hlmax}) is applied in the third inequality (with a different countable index set $\mathcal{D}$). \qedhere
\end{proof}

\subsection{Proof of Theorem \ref{mainequiv}}

One direction follows immediately from Lemma \ref{character2} (2). 
Therefore, we need to prove that there exists a measurable subset $S_Q$ such that $|S_Q|>(1-\gamma)|Q|$ and 
\begin{align}\label{ourmaingoal}
&\Big\Vert \Big\{  \sum_{Q\in\mathcal{D}_k}{\Big( \inf_{y\in Q}{\mathfrak{M}_{\sigma,2^k}^tf_k(y)}\Big)\chi_{S_Q}}   \Big\}_{k\geq \mu} \Big\Vert_{L^{\infty}(l^q)}\nonumber\\
&\lesssim \sup_{P\in \mathcal{D},l(P)\leq 2^{-\mu}}{\Big(\frac{1}{|P|}\int_P{\sum_{k=-\log_2{l(P)}}^{\infty}{|f_k(x)|^q}}dx \Big)^{1/q}}.
\end{align}

To choose such a subset $S_Q$ we set up notation and terminology.
For $0<q\leq \infty$ and $P\in\mathcal{D}$
we define
\begin{equation*}
G^q_P(\mathbf{f})(x):=\Big\Vert \Big\{\sum_{Q\in\mathcal{D}_k, Q\subset P}{\Big( \inf_{y\in Q}{|f_k(y)|}\Big) \chi_Q(x)} \Big\}_{k\geq -\log_2{l(P)}}\Big\Vert_{l^{q}}.
\end{equation*} 
Recall that the nonincreasing rearrangement $f^*$ of a non-negative measurable function $f$ is given by 
\begin{equation*}
f^*(\gamma):=\inf{\big\{\lambda>0:\big| \{x\in\mathbb{R}^d:f(x)>\lambda\}\big|\leq \gamma \big\}}
\end{equation*} and satisfies
\begin{equation}\label{arrangeproperty}
\big| \big\{x\in\mathbb{R}^d:f(x)>f^*(\gamma) \big\}\big|\leq \gamma, \quad \gamma>0.
\end{equation}
For $P\in\mathcal{D}$, $0<\gamma<1$, and a non-negative measurable function $f$,  the ``$\gamma$-median of $f$ over $P$"  is defined as
\begin{equation*}
m_P^{\gamma}(f):=\inf{\big\{ \lambda>0:\big|\big\{x\in P: f(x)>\lambda \big\} \big|\leq\gamma |P|\big\}}.
\end{equation*} 
We consider the $\gamma$-median of $G_P^q(\mathbf{f})$ over $P$ and the supremums of the quantity over $P\in\mathcal{D}$, $l(P)\leq 2^{-\mu}$. 
That is,
\begin{equation*}
\mathbf{m}_P^{\gamma,q}(\mathbf{f}):=m_P^{\gamma}\big(G_P^q(\mathbf{f})\big)=\inf{\big\{ \lambda>0: \big| \big\{x\in P:G_P^q(\mathbf{f})(x)>\lambda\big\}\big|\leq \gamma {|P|} \big\}},
\end{equation*}
\begin{equation*}
\mathbf{m}^{\gamma,q,\mu}(\mathbf{f})(x):=\sup_{P\in\mathcal{D}, l(P)\leq 2^{-\mu}}{\mathbf{m}_P^{\gamma,q}(\mathbf{f})\chi_P(x)}.
\end{equation*}

Observe that 
\begin{equation*}
\mathbf{m}_P^{\gamma,q}(\mathbf{f})=\big( G_P^q(\mathbf{f})\chi_P\big)^*(\gamma |P|)
\end{equation*} and by (\ref{arrangeproperty}) one has
\begin{align}\label{arrangeproperty2}
&\big| \big\{x\in P:G_P^q(\mathbf{f})(x)>\mathbf{m}^{\gamma,q,-\log_2{l(P)}}(\mathbf{f})(x) \big\}\big|\nonumber\\
&\leq \big| \big\{x\in P:G_P^q(\mathbf{f})(x)>\mathbf{m}_P^{\gamma,q}(\mathbf{f}) \big\}\big|\leq \gamma|P|. 
\end{align}
Moreover, 
\begin{equation}\label{membedding}
\mathbf{m}^{\gamma,q,\mu_1}(\mathbf{f})(x)\leq \mathbf{m}^{\gamma,q,\mu_2}(\mathbf{f})(x) \quad \text{for}\quad \mu_1\geq \mu_2.
\end{equation}

Now for each $P\in \mathcal{D}$ we define
\begin{equation*}
S_P^{\gamma,q}(\mathbf{f}):=\big\{ x\in P: G_P^q(\mathbf{f})(x)\leq \mathbf{m}^{\gamma,q,-\log_2{l(P)}}\big(\mathbf{f}\big)(x) \big\}.
\end{equation*}
Then (\ref{arrangeproperty2}) yields that
\begin{equation}\label{claim1}
\big|S_P^{\gamma,q}(\mathbf{f})\big|\geq (1-\gamma)|P|
\end{equation} 
and (\ref{ourmaingoal}) can be deduced in the following proposition.
\begin{proposition}\label{character4}
Let $0<q<\infty$, $\sigma>d/t>d/q$, $0<\gamma<1$, and $\mu\in\mathbb{Z}$. Suppose $A>0$ and $f_k\in\mathcal{E}(A2^k)$ for each $k\in\mathbb{Z}$. 
Then
\begin{align*}
 &\Big\Vert \Big\{ \sum_{Q\in\mathcal{D}_k}{\Big(\inf_{y\in Q}{\mathfrak{M}_{\sigma,2^k}^tf_k(y)}\Big)\chi_{S_Q^{\gamma,q}(\mathfrak{M}_{\sigma}^t(\ff)) }}\Big\}_{k\geq \mu}\Big\Vert_{L^{\infty}(l^q)}\\
 &\lesssim \sup_{P\in\mathcal{D},l(P)\leq 2^{-\mu}}{\Big(\frac{1}{|P|}\int_P{\sum_{k=-\log_2{l(P)}}^{\infty}{{|f_k(x)|^q}}}dx \Big)^{1/q}}
\end{align*} uniformly in $\mu$.
\end{proposition}

\begin{remark}
For $0<q<\infty$
\begin{equation*}
\Vert \ff\Vert_{L^{\infty}(l^q)}\not\approx \Big\Vert \Big\{ \sum_{Q\in\mathcal{D}_k}{\Big(\inf_{y\in Q}{ \mathfrak{M}_{\sigma,2^k}^tf_k(y)}\Big)\chi_{S_Q^{\gamma,q}(\mathfrak{M}_{\sigma}^t(\ff)) }}\Big\}_{k\in\mathbb{Z}}\Big\Vert_{L^{\infty}(l^q)}
\end{equation*} 
while for $0<p<\infty$ or $p=q=\infty$
\begin{equation*}
\Vert \ff\Vert_{L^p(l^q)}\approx \Big\Vert \Big\{ \sum_{Q\in\mathcal{D}_k}{\Big(\inf_{y\in Q}{ \mathfrak{M}_{\sigma,2^k}^tf_k(y)}\Big)\chi_{S_Q^{\gamma,q}(\mathfrak{M}_{\sigma}^t(\ff)) }}\Big\}_{k\in\mathbb{Z}}\Big\Vert_{L^p(l^q)},
\end{equation*}  which is due to Lemma \ref{character2} (1).
\end{remark}

\begin{proof}[Proof of Proposition \ref{character4}]
Assume $0<q<\infty$, $\sigma>d/t>d/q$, and $\mu\in\mathbb{Z}$.
Our claim is
\begin{align}
&\Big\Vert \Big\{ \sum_{Q\in\mathcal{D}_k}{\Big( \inf_{y\in Q}{\mathfrak{M}_{\sigma,2^k}^tf_k(y)}\Big) \chi_{S_Q^{\gamma,q}(\mathfrak{M}_{\sigma}^t(\ff))}}\Big\}_{k\geq \mu}      \Big\Vert_{L^{\infty}(l^q)}\nonumber\\
&= \sup_{P\in\mathcal{D}, l(P)\leq 2^{-\mu}} \Big\Vert \Big\{ \sum_{Q\in\mathcal{D}_k,Q\subset P}{\Big(  \inf_{y\in Q}{ \mathfrak{M}_{\sigma,2^k}^tf_k(y)}\Big)\chi_{S_Q^{\gamma,q}(\mathfrak{M}_{\sigma}^t(\ff)) }}\Big\}_{k\geq \mu}\Big\Vert_{L^{\infty}(l^q)}\nonumber\\
&\leq \big\Vert \mathbf{m}^{\gamma,q,\mu}\big(\mathfrak{M}_{\sigma}^t(\ff) \big) \big\Vert_{L^{\infty}}\label{mclaim1}\tag{Claim 1}\\
&\lesssim  \sup_{P\in\mathcal{D},l(P)\leq 2^{-\mu}}{\Big(\frac{1}{|P|}\int_P{\sum_{k=-\log_2{l(P)}}^{\infty}{{|f_k(x)|^q}}}dx \Big)^{1/q}}\label{mclaim2}\tag{Claim 2}.    
\end{align}

To verify (\ref{mclaim1}) let $\nu\geq \mu$ and fix $P\in\mathcal{D}_{\nu}$ (i.e. $l(P)=2^{-\nu}\leq 2^{-\mu}$). Suppose $x\in P$.
Then it suffices to show that
\begin{equation}\label{claim22}
\Big\Vert \Big\{\sum_{Q\in\mathcal{D}_k,Q\subset P}{\Big(\inf_{y\in Q}{\mathfrak{M}_{\sigma,2^k}^tf_k(y)}\Big)\chi_{S_Q^{\gamma,q}(\mathfrak{M}_{\sigma}^t(\mathbf{f}))}(x)}\Big\}_{k\geq \mu}\Big\Vert_{l^q}\leq \mathbf{m}^{\gamma,q,\nu}\big(\mathfrak{M}_{\sigma}^t(\mathbf{f})\big)(x)
\end{equation}
due to (\ref{membedding}).
Suppose that the left hand side of (\ref{claim22}) is a nonzero number. Then there exists the ``maximal" dyadic cube $P_0(x)\subset P$ such that $x\in S_{P_0(x)}^{\gamma,q}(\mathfrak{M}_{\sigma}^t(\mathbf{f}))$, and thus 
\begin{equation}\label{gpqx}
G_{P_0(x)}^q\big(\mathfrak{M}_{\sigma}^t(\mathbf{f})\big)(x)\leq \mathbf{m}^{\gamma,q,-\log_2{l(P_0(x))}}\big(\mathfrak{M}_{\sigma}^t(\mathbf{f})\big)(x)\leq \mathbf{m}^{\gamma,q,\nu}\big(\mathfrak{M}_{\sigma}^t(\mathbf{f})\big)(x)
\end{equation} where the second inequality follows from (\ref{membedding}).
The maximality of $P_0(x)$ yields that the left hand side of (\ref{claim22}) is
\begin{align*}
&\Big\Vert \Big\{\sum_{Q\in\mathcal{D}_k, Q\subset P_0(x)}{\Big(\inf_{y\in Q}{\mathfrak{M}_{\sigma,2^k}^tf_k(y)}\Big)\chi_{S_Q^{\gamma,q}(\mathfrak{M}_{\sigma}^t(\mathbf{f}))}(x)}\Big\}_{k\geq -\log_2{l(P_0(x))}}  \Big\Vert_{l^q}\\
&\leq \Big\Vert \Big\{\sum_{Q\in\mathcal{D}_k, Q\subset P_0(x)}{\Big(\inf_{y\in Q}{\mathfrak{M}_{\sigma,2^k}^{t}f_k(y)}\Big)\chi_Q(x)}\Big\}_{k\geq -\log_2{l(P_0(x))}}  \Big\Vert_{l^q}\\
&=G_{P_0(x)}^q\big(\mathfrak{M}_{\sigma}^t(\mathbf{f})\big)(x)\leq \mathbf{m}^{\gamma,q,\nu}\big(\mathfrak{M}_{\sigma}^t(\mathbf{f})\big)(x),
\end{align*} where the last one follows from (\ref{gpqx}). This proves (\ref{claim22}).

We now prove (\ref{mclaim2}) for ``any $0<\gamma<\infty$".
Fix $\nu\geq \mu$ and let us assume 
\begin{equation}\label{givenest}
\epsilon>\gamma^{-1/q} \sup_{P\in\mathcal{D}, l(P)\leq 2^{-\nu}}{\Big( \frac{1}{|P|}\int_P{\sum_{k=-\log_2{l(P)}}^{\infty}{|f_k(x)|^q}}dx\Big)^{1/q}}.
\end{equation} Then,  using Chebyshev's inequality, (\ref{qequiv2}), and (\ref{givenest}),  there exists a constant $C_{A,q,t,\sigma}>0$ such that for $R\in \mathcal{D}_{\nu}$
\begin{align*}
&\big| \big\{x\in R:G_R^q\big(\mathfrak{M}_{\sigma}^t(\mathbf{f})\big)(x)>\epsilon \big\}\big|\leq  \frac{1}{\epsilon^q}\big\Vert G_R^q\big( \mathfrak{M}_{\sigma}^t(\mathbf{f})\big)\big\Vert_{L^q}^q\\
&=\frac{1}{\epsilon^q}\int_{R}{\sum_{k=-\log_2{l(R)}}^{\infty}{\Big(\sum_{Q\in\mathcal{D}_k}{\Big(\inf_{y\in Q}{\mathfrak{M}_{\sigma,2^k}^tf_k(y)} \Big)\chi_Q(x)} \Big)^q}}dx\\
&\leq C_{A,q,t,\sigma} \frac{|R|}{\epsilon^q}    \sup_{P\in\mathcal{D}, l(P)\leq 2^{-\nu}}{\Big( \frac{1}{|P|}\int_P{\sum_{k=-\log_2{l(P)}}^{\infty}{|f_k(x)|^q}}dx\Big)^{1/q}}    \\
&\leq C_{A,q,t,\sigma}\gamma |R|.
\end{align*}
This yields that
\begin{equation*}
\mathbf{m}_R^{C_{A,q,t,\sigma}\gamma,q}\big(\mathfrak{M}_{\sigma}^t(\mathbf{f})\big)\leq \epsilon <2\epsilon.
\end{equation*}
So far, we have proved that for any $R\in\mathcal{D}_{\nu}$,
\begin{equation*}
\mathbf{m}_R^{C_{A,q,t,\sigma}\gamma,q}\big(\mathfrak{M}_{\sigma}^t(\mathbf{f})\big)\leq 2\gamma^{-1/q}\sup_{P\in\mathcal{D}, l(P)\leq 2^{-\nu}}{\Big( \frac{1}{|P|}\int_P{\sum_{k=-\log_2{l(P)}}^{\infty}{|f_k(x)|^q}}dx\Big)^{1/q}},
\end{equation*}
which is equivalent to
\begin{equation*}
\mathbf{m}_R^{\gamma,q}\big(\mathfrak{M}_{\sigma}^t(\mathbf{f})\big)\leq 2C_{A,q,t,\sigma}^{1/q}\gamma^{-1/q}\sup_{P\in\mathcal{D}, l(P)\leq 2^{-\nu}}{\Big( \frac{1}{|P|}\int_P{\sum_{k=-\log_2{l(P)}}^{\infty}{|f_k(x)|^q}}dx\Big)^{1/q}}.
\end{equation*}
We complete the proof by taking the supremum over $R\in\mathcal{D}, l(R)=2^{-\nu}\leq 2^{-\mu}$. \qedhere
\end{proof}

We end this section by pointing out  that the replacement of Lemma \ref{maximal2} in the above arguments by Lemma \ref{maximal3} provides the following corollaries.
\begin{corollary}\label{cor1}
Let $0<p,q\leq\infty$, $\min{(p,q)}\leq t\leq \infty$, $\sigma>d/\min{(p,q)}$, $0<\gamma< 1$, and $\mu\in\mathbb{Z}$. 
For each $Q\in\mathcal{D}$ let $S_Q$ be a measurable subset of $Q$ with $|S_Q|>(1-\gamma) |Q|$.
Suppose $A>0$ and $f_k\in \mathcal{E}(A2^k)$ for each $k\in\mathbb{Z}$.
\begin{enumerate}
\item For $0<p<\infty$ or $p=q=\infty$
\begin{align*}
\Vert \{f_k \}_{k\in\mathbb{Z}}\Vert_{L^p(l^q)}&\approx \Big\Vert \Big\{ \sum_{Q\in\mathcal{D}_k}{\Big( \inf_{y\in Q}{\mathfrak{M}_{\sigma,2^k}^tf_k(y)}\Big)\chi_{Q}}\Big\}_{k\in\mathbb{Z}}\Big\Vert_{L^p(l^q)}\\
 &\approx_{\gamma} \Big\Vert \Big\{ \sum_{Q\in\mathcal{D}_k}{\Big( \inf_{y\in Q}{\mathfrak{M}_{\sigma,2^k}^tf_k(y)}\Big)\chi_{S_Q}}\Big\}_{k\in\mathbb{Z}}\Big\Vert_{L^p(l^q)}.
\end{align*}
\item For $p=\infty$ and $0<q<\infty$
\begin{align*}
&\sup_{P\in \mathcal{D},l(P)\leq 2^{-\mu}}{\Big(\frac{1}{|P|}\int_P{\sum_{k=-\log_2{l(P)}}^{\infty}{|f_k(x)|^q}}dx \Big)^{1/q}}\\
&\approx  \sup_{P\in \mathcal{D},l(P)\leq 2^{-\mu}}{\Big(\frac{1}{|P|}\int_P{\sum_{k=-\log_2{l(P)}}^{\infty}{\Big|  \sum_{Q\in\mathcal{D}_k}{\Big( \inf_{y\in Q}{\mathfrak{M}_{\sigma,2^k}^tf_k(y)}\Big)\chi_{Q}(x)}   \Big|^q}}dx \Big)^{1/q}}\\
&\approx  \sup_{P\in \mathcal{D},l(P)\leq 2^{-\mu}}{\Big(\frac{1}{|P|}\int_P{\sum_{k=-\log_2{l(P)}}^{\infty}{\Big|  \sum_{Q\in\mathcal{D}_k}{\Big( \inf_{y\in Q}{\mathfrak{M}_{\sigma,2^k}^tf_k(y)}\Big)\chi_{S_Q}(x)}   \Big|^q}}dx \Big)^{1/q}}.
\end{align*}
\end{enumerate}
\end{corollary}
\begin{corollary}\label{cor2}
Let $0<q<\infty$, $q\leq t\leq \infty$, $\sigma>d/q$, $0<\gamma<1$, and $\mu\in\mathbb{Z}$. Suppose $A>0$ and $f_k\in \mathcal{E}(A2^k)$ for each $k\in \mathbb{Z}$. For  $Q\in\mathcal{D}$ there exists a proper measurable subset $S_Q$ of $Q$, depending on $\gamma,q,\sigma,t,\{f_k\}_{k\in\mathbb{Z}}$, such that $|S_Q|>(1-\gamma)|Q|$ and
\begin{align*}
&\sup_{P\in \mathcal{D},l(P)\leq 2^{-\mu}}{\Big(\frac{1}{|P|}\int_P{\sum_{k=-\log_2{l(P)}}^{\infty}{|f_k(x)|^q}}dx \Big)^{1/q}}\\
&\relphantom{=} \approx \Big\Vert \Big\{  \sum_{Q\in\mathcal{D}_k}{\Big( \inf_{y\in Q}{\mathfrak{M}_{\sigma,2^k}^tf_k(y)}\Big)\chi_{S_Q}}   \Big\}_{k\geq \mu} \Big\Vert_{L^{\infty}(l^q)}, \quad \text{uniformly in }~\mu.
\end{align*}
\end{corollary}

\section{Proof of Theorem \ref{multipliermain0} and \ref{multipliermain}}\label{multimultipliersection}

We will first prove Theorem \ref{multipliermain} and then turn to the proof of Theorem \ref{multipliermain0}.
Before proving the theorem we set up some notation.
Write  $\fff:=(f_1,\dots,f_n)$,  $\xxi:=(\xi_1,\dots,\xi_n)$,   $\vv:=(v_1,\dots,v_n)$, $d\xxi:=d\xi_1\cdots d\xi_n$, and $d\vv:=dv_1\cdots dv_n$.
\subsection{Proof of Theorem \ref{multipliermain}}
Choose $0<t<r$ such that $s>nd/t> nd/r(\geq nd/2)$ and let $\delta=r-t>0$. Suppose $p_1,\dots,p_n>t=r-\delta$.
Then
\begin{equation}\label{sncondition}
s/n>d/t>d/\min{(2,p_1,p_2,\dots,p_n)}.
\end{equation}

Let $\widetilde{\vartheta^{(n)}}$ be a cutoff function on $(\mathbb{R}^d)^n$ such that $0\leq \widetilde{\vartheta^{(n)}}\leq 1$, $\widetilde{\vartheta^{(n)}}(\xxi)=1$ for $2^{-2}n^{-1/2}\leq |\xxi|\leq 2n^{1/2}$, and 
$Supp(\widetilde{\vartheta^{(n)}})\subset \big\{ \xxi \in(\mathbb{R}^d)^n: 2^{-3}n^{-1/2}\leq |\xxi|\leq 2^{2}n^{1/2}\big\}$.
Then using Calder\'on's reproducing formula, Littlewood-Paley partition of unity $\{\phi_k\}_{k\in\mathbb{Z}}$, and triangle inequality, we first see that
\begin{equation*}
  \mathcal{L}_{s}^{u,\widetilde{\vartheta^{(n)}}}[m]\lesssim  \mathcal{L}_s^{u,\vartheta^{(n)}}[m].
\end{equation*}
Thus it suffices to prove the estimate that
\begin{equation}\label{finalgoal}
\big\Vert T_m\fff \big\Vert_{L^p}\lesssim \mathcal{L}_s^{u,\widetilde{\vartheta^{(n)}}}[m] \sum_{i=1}^{n}{\Big(\Vert f_i\Vert_{X^{p_{i,i}}}\prod_{1\leq j\leq n, j\not= i}{\Vert f_j\Vert_{H^{p_{i,j}}}} \Big)}.
\end{equation}
We use a notation $\mathcal{L}_s^u[m]:=  \mathcal{L}_{s}^{u,\widetilde{\vartheta^{(n)}}}[m] $.

By using Littlewood-Paley partition of unity $\{\phi_k\}_{k\in\mathbb{Z}}$, $m(\xxi)$ can be decomposed as
\begin{align*}
m(\xxi)&=\sum_{k_1,\dots,k_n \in \mathbb{Z}}{m(\xxi)\widehat{\phi_{k_1}}(\xi_1)\cdots \widehat{\phi_{k_n}}(\xi_n)}\\
  &=\Big(\sum_{k_1\in\mathbb{Z}}\sum_{k_2,\dots,k_n\leq k_1}{\cdots}\Big)+\Big(\sum_{k_2\in\mathbb{Z}}\sum_{\substack{k_1<k_2\\k_3,\dots,k_n\leq k_2}}{\cdots}\Big)+\dots+ \Big(\sum_{k_n\in\mathbb{Z}}\sum_{k_1,\dots,k_{n-1}<k_n}{\cdots}\Big)\\
  &=:m^{(1)}(\xxi)+m^{(2)}(\xxi)+\dots+m^{(n)}(\xxi).
\end{align*}
Then (\ref{finalgoal}) is a consequence of the following estimates that
\begin{equation*}
\big\Vert T_{m^{(i)}}\fff \big\Vert_{L^p}\lesssim  \mathcal{L}_s^u[m]\Vert f_i\Vert_{X^{p_{i,i}}}\prod_{{1\leq j\leq n, j\not= i}}\Vert f_j\Vert_{H^{p_{i,j}}}
\end{equation*} for each $1\leq i\leq n$.
We only concern ourselves with the case $i=1$ by setting  $p_{j}:=p_{1,j}$ for $1\leq j\leq n$, and use symmetry for other cases. Suppose $1/p=1/p_1+\dots +1/p_n$.

We write \begin{align*}
m^{(1)}(\xxi)&=\sum_{k\in\mathbb{Z}}\sum_{k_2,\dots,k_n\leq k}{m(\xxi)\widehat{\phi_k}(\xi_1)\widehat{\phi_{k_2}}(\xi_2)\cdots\widehat{\phi_{k_n}}(\xi_n)}\\
 &=\sum_{k\in\mathbb{Z}}{m(\xxi)\widetilde{\vartheta^{(n)}}(\xxi/2^k) \widehat{\phi_k}(\xi_1)\sum_{k_2,\dots,k_n\leq k}\widehat{\phi_{k_2}}(\xi_2)\cdots\widehat{\phi_{k_n}}(\xi_n)}
\end{align*} 
since $\widetilde{\vartheta^{(n)}}(\xxi/2^k)=1$ for $ 2^{k-2}\leq |\xi_1|\leq 2^{k}$ and $|\xi_j|\leq 2^{k}$ for $2\leq j\leq n$.
Let 
\begin{equation*}
m_k(\xxi):=m(\xxi)\widetilde{\vartheta^{(n)}}(\xxi/2^k). 
\end{equation*}
Then we note that
\begin{equation}\label{mkbound}
\big\Vert m_k(2^k\cdot )\big\Vert_{L_s^u((\mathbb{R}^d)^n)}\leq \mathcal{L}_s^u[m]
\end{equation} 
and 
\begin{equation*}
m^{(1)}(\xxi)=\sum_{k\in\mathbb{Z}}{m_k(\xxi) \widehat{\phi_k}(\xi_1)\sum_{k_2,\dots,k_n\leq k}\widehat{\phi_{k_2}}(\xi_2)\cdots\widehat{\phi_{k_n}}(\xi_n)}.
\end{equation*}
We further decompose $m^{(1)}$ as
\begin{equation*}
m^{(1)}(\xxi)=m^{(1)}_{low}(\xxi)+m^{(1)}_{high}(\xxi)
\end{equation*} where
\begin{equation*}
m^{(1)}_{low}(\xxi):=\sum_{k\in\mathbb{Z}}m_k(\xxi)\widehat{\phi_k}(\xi_1)\sum_{\substack{k_2,\dots,k_n\leq k\\ \max_{2\leq j\leq n}{(k_j)}\geq k-3-\lfloor \log_2{n}\rfloor}}{\widehat{\phi_{k_2}}(\xi_2)\cdots\widehat{\phi_{k_n}}(\xi_n)},
\end{equation*}
\begin{equation*}
m^{(1)}_{high}(\xxi):=\sum_{k\in\mathbb{Z}}m_k(\xxi)\widehat{\phi_k}(\xi_1)\sum_{k_2,\dots,k_n\leq k-4-\lfloor \log_2n\rfloor}{\widehat{\phi_{k_2}}(\xi_2)\cdots\widehat{\phi_{k_n}}(\xi_n)}.
\end{equation*}
We refer to $T_{m_{low}^{(1)}}$ as the low frequency part, and $T_{m_{high}^{(1)}}$ as the high frequency part of $T_{m^{(1)}}$ (due to the Fourier supports of $T_{m_{low}^{(1)}}\fff$ and $T_{m_{high}^{(1)}}\fff$).

\subsubsection{Low frequency part}
To obtain the estimates for the operator $T_{m^{(1)}_{low}}$,
we observe that
\begin{equation*}
T_{m^{(1)}_{low}}\fff(x)=\sum_{k\in\mathbb{Z}}\sum_{\substack{k_2,\dots,k_n\leq k\\ \max_{2\leq j\leq n}{(k_j)}\geq k-3-\lfloor \log_2{n}\rfloor}} {T_{m_k}\big( (f_1)_k,(f_2)_{k_2},\dots,(f_n)_{k_n}\big)(x)}
\end{equation*}
where $(g)_l:=\phi_l\ast g$ for $g\in S$ and $l\in\mathbb{Z}$.
 It suffices to treat only the sum over $k_3,\dots,k_n\leq k_2$ and $k-3-\lfloor \log_2{n}\rfloor \leq k_2\leq k$, and we will actually prove that
\begin{align*}
&\Big\Vert \sum_{k\in\mathbb{Z}}\sum_{k-3-\lfloor \log_2{n}\rfloor \leq k_2\leq k}\sum_{k_3,\dots,k_n\leq k_2}T_{m_k}\big((f_1)_k,(f_2)_{k_2},\dots,(f_n)_{k_n} \big)\Big\Vert_{L^p}\nonumber\\
&\lesssim \mathcal{L}_s^u[m]\Vert f_1\Vert_{X^{p_1}}\prod_{j=2}^{n}\Vert f_j\Vert_{H^{p_j}}. 
\end{align*}

We define $\Phi_l:=2^{ld}\Phi_0(2^l\cdot)$ for $l\in\mathbb{Z}$ as before, and then observe that for any $g\in S$
\begin{equation*}
\sum_{m\leq l}\phi_m\ast g=\Phi_l\ast g.
\end{equation*}
Then
\begin{align*}
&\sum_{k\in\mathbb{Z}}\sum_{k-3-\lfloor \log_2{n}\rfloor \leq k_2\leq k}\sum_{k_3,\dots,k_n\leq k_2}T_{m_k}\big((f_1)_k,(f_2)_{k_2},\dots,(f_n)_{k_n} \big)(x)\\
&=\sum_{k\in\mathbb{Z}}\sum_{k-3-\lfloor \log_2{n}\rfloor \leq k_2\leq k}T_{m_k}\big((f_1)_k,(f_2)_{k_2},(f_3)^{k_2},\dots,(f_n)^{k_2} \big)(x)
\end{align*}
where  $(f_j)^{k_2}:=\Phi_{k_2}\ast f_j$.  
Since the second sum is a finite sum over $k_2$ near $k$, we may only consider the case $k_2=k$ and thus our claim is 
\begin{equation}\label{mainmainclaim}
\Big\Vert \sum_{k\in\mathbb{Z}}{T_{m_k}\big((f_1)_k,(f_2)_k,(f_3)^k,\dots,(f_n)^k \big)}\Big\Vert_{L^p}\lesssim \mathcal{L}_s^u[m]\Vert f_1\Vert_{X^{p_1}}\prod_{j=2}^{n}{\Vert f_j\Vert_{H^{p_j}}}.
\end{equation} 

To prove (\ref{mainmainclaim}) let $0<\epsilon<\min{(1,t)}$ such that $1/\epsilon=1-1/u+1/t$,
which implies $u'=\frac{1}{1/\epsilon-1/t}$ where $1/u+1/u'=1$.
Then using Nikolskii's inequality and H\"older's inequality with $u'/\epsilon>1$ one has
\begin{align*}
&\big|T_{m_k}\big((f_1)_k,(f_2)_k,(f_3)^k,\dots,(f_n)^k \big)(x) \big|\\
&=\Big|\int_{(\mathbb{R}^d)^n}{m_k^{\vee}(\vv)(f_1)_k(x-v_1)(f_2)_k(x-v_2)\prod_{j=3}^{n}(f_j)^k(x-v_j)}d\vv \Big|\\
&\lesssim 2^{nkd(1/\epsilon-1)}\Big( \int_{(\mathbb{R}^d)^n}{|m_k^{\vee}(\vv)|^{\epsilon}\big|(f_1)_k(x-v_1)\big|^{\epsilon}\big|(f_2)_k(x-v_2)\big|^{\epsilon}\prod_{j=3}^{n}\big| (f_j)^k(x-v_j)\big|^{\epsilon}}d\vv \Big)^{1/\epsilon}\\
&\leq 2^{nkd(1/\epsilon-1)}\Big( \int_{(\mathbb{R}^d)^n}{\big(1+2^k|v_1|+\cdots+2^k|v_n| \big)^{su'}|m_k^{\vee}(\vv)|^{u'}}d\vv \Big)^{1/u'}\\
 &\relphantom{=}\times \Big( \int_{(\mathbb{R}^d)^n}{\frac{\big|(f_1)_k(x-v_1) \big|^t\big|(f_2)_k(x-v_2) \big|^t}{\big( 1+2^k|v_1|+\cdots +2^k|v_n|\big)^{st}}\prod_{j=3}^{n}{\big| (f_j)^k(x-v_n)\big|^t}}d\vv\Big)^{1/t}.
\end{align*}
By using Hausdorff Young's inequality with $u'\geq 2$ and (\ref{mkbound}),
\begin{align*}
&\Big( \int_{(\mathbb{R}^d)^n}{\big(1+2^k|v_1|+\cdots+2^k|v_n| \big)^{su'}|m_k^{\vee}(\vv)|^{u'}}d\vv \Big)^{1/u'}\lesssim 2^{nkd/u} \big\Vert m_k(2^k\cdot )\big\Vert_{L_s^u}\lesssim 2^{nkd/u}\mathcal{L}_s^u[m],
\end{align*}
and
\begin{align*}
& \Big( \int_{(\mathbb{R}^d)^n}{\frac{\big|(f_1)_k(x-v_1) \big|^t\big|(f_2)_k(x-v_2) \big|^t}{\big( 1+2^k|v_1|+\cdots +2^k|v_n|\big)^{st}}\prod_{j=3}^{n}{\big| (f_j)^k(x-v_n)\big|^t}}d\vv\Big)^{1/t}\\
&\leq \Big( \int_{\mathbb{R}^d}{\frac{|(f_1)_k(x-v_1)|^t}{(1+2^k|v_1|)^{st/n}}}dv_1\Big)^{1/t}\Big( \int_{\mathbb{R}^d}{\frac{|(f_2)_k(x-v_2)|^t}{(1+2^k|v_2|)^{st/n}}}dv_2\Big)^{1/t}\prod_{j=3}^{n}\Big( \int_{\mathbb{R}^d}{\frac{|(f_j)^k(x-v_j)|^t}{(1+2^k|v_j|)^{st/n}}}dv_j\Big)^{1/t}\\
&\leq 2^{-nkd/t}\mathfrak{M}_{s/n,2^k}^{t}(f_1)_k(x)\mathfrak{M}_{s/n,2^k}^{t}(f_2)_k(x)\prod_{j=3}^{n}\mathfrak{M}_{s/n,2^k}^{t}(f_j)^k(x).
\end{align*}
Therefore
\begin{align}\label{keyestimate}
&\big|T_{m_k}\big((f_1)_k,(f_2)_k,(f_3)^k,\dots,(f_n)^k \big)(x) \big|\nonumber\\
&\lesssim \mathcal{L}_s^u[m]\mathfrak{M}_{s/n,2^k}^{t}(f_1)_k(x)\mathfrak{M}_{s/n,2^k}^{t}(f_2)_k(x)\prod_{j=3}^{n}\mathfrak{M}_{s/n,2^k}^{t}(f_j)^k(x)
\end{align}
because $1/\epsilon-1+1/u-1/t=0$.

Let 
\begin{equation*}
S_Q^{(1)}:=S_Q^{1/4,2}\big(\mathfrak{M}_{s/n,2^k}^t(f_1)_k \big)\quad \text{ and } \quad S_Q^{(2)}:=S_Q^{1/4,2}\big(\mathfrak{M}_{s/n,2^k}^t(f_2)_k \big).
\end{equation*}
Then it follows from (\ref{claim1}) that $S_Q^{(1)}$ and $S_Q^{(2)}$ are measurable subsets of $Q$ such that
$|S_Q^{(1)}|, |S_Q^{(2)}|\geq \frac{3}{4}|Q|$.
We observe that $|S_Q^{(1)}\cap S_Q^{(2)}|\geq \frac{1}{2}|Q|$ and thus, for any $\tau>0$
\begin{equation}\label{intersectionchi}
\chi_Q(x)\lesssim_{\tau}\mathcal{M}_{\tau}\big(\chi_{S_Q^{(1)}\cap S_Q^{(2)}}\big)(x)\chi_Q(x),
\end{equation} using the argument in (\ref{inversechi}).
Clearly, the constant in the inequality is independent of $Q$.

Now we choose $\tau<\min{(1,p)}$, and  apply (\ref{keyestimate}), (\ref{infmax2}), (\ref{intersectionchi}), and (\ref{hlmax}) to obtain
\begin{align*}
&\Big\Vert \sum_{k\in\mathbb{Z}}{T_{m_k}\big((f_1)_k,(f_2)_k,(f_3)^k,\dots,(f_n)^k \big)}\Big\Vert_{L^p}\\
&\lesssim \mathcal{L}_s^u[m]\Big\Vert \sum_{k\in\mathbb{Z}}{\sum_{Q\in \mathcal{D}_k}{{\mathfrak{M}_{s/n,2^k}^t(f_1)_k} {\mathfrak{M}_{s/n,2^k}^t(f_2)_k} \prod_{j=3}^{n}{\mathfrak{M}_{s/n,2^k}^t(f_j)^k}  \chi_Q  }}\Big\Vert_{L^{p}}\\
&\lesssim \mathcal{L}_s^u[m]  \Big\Vert \sum_{k\in\mathbb{Z}}{\sum_{Q\in \mathcal{D}_k}{\Big[\prod_{i=1}^{2}\Big( \inf_{y\in Q}{\mathfrak{M}_{s/n,2^k}^t(f_i)_k(y)}\Big)\Big]\Big[ \prod_{j=3}^{n}\Big( \inf_{y\in Q}{\mathfrak{M}_{s/n,2^k}^t(f_j)^k(y)}\Big)\Big]  \mathcal{M}_{\tau}{(\chi_{S_Q^{(1)}\cap S_Q^{(2)}})  }}}\Big\Vert_{L^{p}}\\
&\lesssim \mathcal{L}_s^u[m]  \Big\Vert \sum_{k\in\mathbb{Z}}{\sum_{Q\in \mathcal{D}_k}{\Big[\prod_{i=1}^{2}\Big( \inf_{y\in Q}{\mathfrak{M}_{s/n,2^k}^t(f_i)_k(y)}\Big)\Big]\Big[ \prod_{j=3}^{n}\Big( \inf_{y\in Q}{\mathfrak{M}_{s/n,2^k}^t(f_j)^k(y)}\Big)\Big]  {\chi_{S_Q^{(1)}}  }{\chi_{S_Q^{(2)}}  }}}\Big\Vert_{L^{p}}.
\end{align*}
By using H\"older's inequality, 
the $L^p$ norm is dominated by a constant times
\begin{align*}
& \prod_{i=1}^{2}\Big\Vert \Big\{\sum_{Q\in\mathcal{D}_k}{   \Big( \inf_{y\in Q}{\mathfrak{M}_{s/n,2^k}^t(f_i)_k(y)} \Big)  \chi_{S_Q^{(i)}}    }\Big\}_{k\in\mathbb{Z}}\Big\Vert_{L^{p_i}(l^2)}\prod_{j=3}^{n}{\big\Vert \big\{ \mathfrak{M}_{s/n,2^k}^{t}(f_j)^k\big\}_{k\in\mathbb{Z}}\big\Vert_{L^{p_j}(l^{\infty})}}\\
&\lesssim \Vert f_1\Vert_{X^{p_1}}\Vert f_2\Vert_{X^{p_2}}\prod_{j=3}^{n}{\Vert f_j\Vert_{H^{p_j}}} 
\end{align*}
where the inequality follows from Lemma \ref{character2} (1), Proposition \ref{character4}, Lemma \ref{maximal2} (1), and (\ref{localhardy}) with (\ref{sncondition}). 
Since $\Vert f_2\Vert_{X^{p_2}}\lesssim \Vert f_2\Vert_{H^{p_2}}$, one finally obtains (\ref{mainmainclaim}).

\subsubsection{High frequency part}
The proof for the high frequency part relies on the fact that if $\widehat{g_k}$ is supported on $\{\xi : C^{-1} 2^{k}\leq |\xi|\leq C2^{k}\}$ for $C>1$ then
\begin{equation}\label{marshall}
\Big\Vert \Big\{ \phi_k\ast \Big(\sum_{l=k-h}^{k+h}{g_l}\Big)\Big\}_{k\in\mathbb{Z}}\Big\Vert_{L^p(l^q)}\lesssim_{h,C} \big\Vert \big\{ g_k\big\}_{k\in\mathbb{Z}}\big\Vert_{L^p(l^q)}
\end{equation} for $h\in \mathbb{N}$. The proof of (\ref{marshall}) is elementary and standard, so it will not pursued here. Just use the estimate $|\phi_k\ast g_l(x)|\lesssim \mathfrak{M}_{\sigma,2^l}g_l(x)$ for $k-h\leq l\leq k+h$ and apply Lemma \ref{maximal2} (1).

We note that
\begin{equation*}
T_{m_{high}^{(1)}}\fff(x)=\sum_{k\in\mathbb{Z}}{T_{m_k}\big((f_1)_k,(f_2)^{k,n},\dots, (f_n)^{k,n} \big)(x)}
\end{equation*}
where $(f_j)^{k,n}:=\Phi_{k-4-\lfloor \log_2{n}\rfloor}\ast f_j$.

Observe that the Fourier transform of $T_{m_k}\big((f_1)_k,(f_2)^{k,n},\dots,(f_n)^{k,n} \big)$ is supported in $\big\{\xi\in\mathbb{R}^d : 2^{k-3}\leq |\xi|\leq 2^{k+2} \big\}$ and thus (\ref{marshall}) yields that
\begin{align*}
\big\Vert T_{m_{high}^{(1)}}\fff \big\Vert_{L^p}&\lesssim \big\Vert T_{m_{high}^{(1)}}\fff \big\Vert_{H^p}\approx \big\Vert T_{m_{high}^{(1)}}\fff \big\Vert_{\dot{F}_p^{0,2}}\\
&\lesssim \big\Vert \big\{  T_{m_k}\big((f_1)_k,(f_2)^{k,n},\dots,(f_n)^{k,n} \big)\big\}_{k\in\mathbb{Z}}\big\Vert_{L^p(l^{2})}.
\end{align*}
Using the argument that led to (\ref{keyestimate}), one has
\begin{equation*}
\big| T_{m_k}\big((f_1)_k,(f_2)^{k,n},\dots,(f_n)^{k,n} \big)(x)\big| \lesssim \mathcal{L}_s^u[m]\mathfrak{M}_{s/n,2^k}^{t}(f_1)_k(x)\prod_{j=2}^{n}\mathfrak{M}_{s/n,2^k}^{t}(f_j)^{k,n}(x).
\end{equation*}
Fix $0<\gamma<1$. For $Q\in \mathcal{D}_k$ let $S_Q:=S_Q^{\gamma,2}\big(\mathfrak{M}_{s/n,2^k}^t(f_1)_k \big)$ as before and proceed the similar arguments to obtain that
\begin{align*}
\big\Vert T_{m_{high}^{(1)}}\fff \big\Vert_{L^p}&\lesssim \mathcal{L}_s^u[m]\Big\Vert  \Big\{ \mathfrak{M}_{s/n,2^k}^{t}(f_1)_k\prod_{j=2}^{n}\mathfrak{M}_{s/n,2^k}^{t}(f_j)^{k,n}\Big\}_{k\in\mathbb{Z}} \Big\Vert_{L^p(l^2)}\\
&\lesssim \mathcal{L}_s^u[m]\Big\Vert  \Big\{ \sum_{Q\in\mathcal{D}_k}{ \Big( \inf_{y\in Q}{\mathfrak{M}_{s/n,2^k}^{t}(f_1)_k(y)}\Big) \Big[\prod_{j=2}^{n}\Big(\inf_{y\in Q}{\mathfrak{M}_{s/n,2^k}^{t}(f_j)^{k,n}(y)\Big) }\Big]\chi_{S_Q}} \Big\}_{k\in\mathbb{Z}} \Big\Vert_{L^p(l^2)}\\
&\lesssim \mathcal{L}_s^u[m]\Big\Vert \Big\{ \sum_{Q\in\mathcal{D}_k}{\Big( \inf_{y\in Q}{\mathfrak{M}_{s/n,2^k}^t(f_1)(y)}\Big)\chi_{S_Q}}\Big\}_{k\in\mathbb{Z}}\Big\Vert_{L^{p_1}(l^2)} \prod_{j=2}^{n}{\Vert f_j\Vert_{H^{p_j}}}\\
&\lesssim \mathcal{L}_s^u[m]\Vert f_1\Vert_{X^{p_1}}\prod_{j=2}^{n}{\Vert f_j\Vert_{H^{p_j}}}.
\end{align*}

\subsection{Proof of Theorem \ref{multipliermain0}}
As in the proof of Theorem \ref{multipliermain}, it suffices to deal with  $T_{m^{(1)}}$.
Suppose $1<p<\infty$ and $1<p_{j}\leq \infty$ for $1\leq j\leq n$.
Then we will prove
\begin{equation}\label{multi0result}
\big\Vert T_{m^{(1)}}\fff \big\Vert_{L^p}\lesssim  \mathcal{L}_s^2[m]\Vert f_i\Vert_{X^{p_{1}}}\prod_{j=2}^{n}\Vert f_j\Vert_{L^{p_{j}}}, \qquad s>nd/2 
\end{equation} for each $1\leq i\leq n$.
First of all, it follows, from Theorem \ref{multipliermain} with $r=u=2$, that
(\ref{multi0result}) holds for $2\leq p_{j}\leq \infty$.

Now assume $1<p\leq \min{(p_{1},\dots,p_{n})}<2$. Observe that only one of $p_j's$ could be less than $2$ because $1/p=1/p_1+\dots+1/p_n<1$, and we will actually look at two cases $1<p_1<2\leq p_2,\dots,p_n$ and $1<p_2<2\leq p_1,p_3,\dots,p_n$.
Let $T_{m^{(1)}}^{*j}$ be the $j$th transpose of $T_{m^{(1)}}$, defined by the unique operator satisfying
\begin{equation*}
\big\langle T_{m^{(1)}}^{*j}(f_1,\dots,f_n),h\big\rangle:=\big\langle T_{m^{(1)}}(f_1,\dots,f_{j-1},h,f_{j+1},\dots,f_n),f_j\big\rangle
\end{equation*}
for $f_1,\dots,f_n,h\in S$.
Then it is known in \cite{Tom} that $T_{m^{(1)}}^{*j}=T_{(m^{(1)})^{*j}}$ where 
\begin{equation*}
(m^{(1)})^{*j}(\xi_1,\dots,\xi_n)=m^{(1)}\big(\xi_1,\dots,\xi_{j-1},-(\xi_1+\dots+\xi_n),\xi_{j+1},\dots,\xi_n\big),
\end{equation*}
and then
\begin{equation}\label{multiest}
\mathcal{L}_s^2\big[(m^{(1)})^{*j}\big]\lesssim \mathcal{L}_s^2[m^{(1)}]\lesssim \mathcal{L}_s^2[m].
\end{equation}

\subsubsection{The case $1<p<p_1<2$}
Let $2<p',p_1'<\infty$ be the conjugates of $p, p_1$, respectively. That is, $1/p+1/p'=1/p_1+1/p_1'=1$. Then $X^{p_1}=L^{p_1}$ and
$1/p_1'=1/p'+1/p_2+\dots+1/p_n$.
Therefore
\begin{align*}
\big\Vert T_{m^{(1)}}(f_1,\dots,f_n)\big\Vert_{L^p} &= \sup_{\Vert h\Vert_{L^{p'}}=1}{\big|\big\langle T_{(m^{(1)})^{*1}}(h,f_2,\dots,f_n),f_1\big\rangle \big|}\\
&\leq \Vert f_1\Vert_{L^{p_1}} \sup_{\Vert h\Vert_{L^{p'}}=1}{\big\Vert T_{(m^{(1)})^{*1}}(h,f_2,\dots,f_n)\big\Vert_{L^{p_1'}} }\\
&\lesssim \mathcal{L}_s^2\big[(m^{(1)})^{*1}\big]\prod_{j=1}^{n}\Vert f_j\Vert_{L^{p_{j}}}\lesssim \mathcal{L}_s^2[m]\prod_{j=1}^{n}\Vert f_j\Vert_{L^{p_{j}}}
\end{align*}
where the second inequality follows from Theorem \ref{multipliermain} and the last one from (\ref{multiest}).
\subsubsection{The case $1<p<p_2<2$}
Similarly, let $2<p',p_2'<\infty$ be the conjugates of $p,p_2$ and then
\begin{align*}
\big\Vert T_{m^{(1)}}(f_1,\dots,f_n)\big\Vert_{L^p} &= \sup_{\Vert h\Vert_{L^{p'}}=1}{\big|\big\langle T_{(m^{(1)})^{*2}}(f_1,h,f_3,\dots,f_n),f_2\big\rangle \big|}\\
&\leq\Vert f_2\Vert_{L^{p_2}} \sup_{\Vert h\Vert_{L^{p'}}=1}{\big\Vert T_{(m^{(1)})^{*2}}(f_1,h,f_3,\dots,f_n)\big\Vert_{L^{p_2'}} }\\
&\lesssim \mathcal{L}_s^2\big[(m^{(1)})^{*2}\big]\Vert f_1\Vert_{X^{p_{1}}}\prod_{j=2}^{n}\Vert f_j\Vert_{L^{p_{j}}}\lesssim \mathcal{L}_s^2[m]\Vert f_1\Vert_{X^{p_{1}}}\prod_{j=2}^{n}\Vert f_j\Vert_{L^{p_{j}}}.
\end{align*}

\section{Proof of Theorem \ref{pseudotheorem} }\label{pseudosection}

We use notations $\fff:=(f_1,\dots,f_n)$,  $\xxi:=(\xi_1,\dots,\xi_n)$, $\ll:=(l_1,\dots,l_n)$, $d\xxi:=d\xi_1\cdots d\xi_n$, $d\eeta:=d\eta_1\cdots d\eta_n$, $\aa:=(\a_1,\dots,\a_n)$, $\|\aa \|:=|\a_1|+\dots+|\a_n|$, $\partial_{\xxi}^{\aa}:=\partial_{\xi_1}^{\alpha_1}\cdots\partial_{\xi_n}^{\alpha_n}$, and $\partial_{\eeta}^{\aa}:=\partial_{\eta_1}^{\alpha_1}\cdots\partial_{\eta_n}^{\alpha_n}$. 

The proof is based on the decomposition technique by B\'enyi-Torres \cite{Be_To}.
Throughout this section we regard $\phi_0=\Phi_0$, not the original meaning $\phi_0=\phi$, so that $\{\phi_k\}_{k\in\mathbb{N}_0}$ is inhomogeneous Littlewood-Paley partition of unity.
We write $\mathcal{A}_N^{(m)}:=\Vert a\Vert_{\MM_n\mathcal{S}_{1,1,N}^{m}} $ for simplicity.
\subsection{Decomposition and reduction}
By using Littlewood-Paley partition of unity, $a\in\MM_n\mathcal{S}_{1,1}^{m}$ can be written as
\begin{align*}
a(x,\xxi)&= \sum_{k_1,\dots, k_n\in\mathbb{N}_0}{a(x,\xxi)\widehat{\phi_{k_1}}(\xi_1)\cdots \widehat{\phi_{k_n}}(\xi_n)}\\
     &= \Big(\sum_{k_2,\dots,k_n\leq k_1}{\cdots}\Big)+\Big(\sum_{\substack{k_3,\dots,k_n\leq k_2\\ k_1<k_2}}{\cdots}\Big)+\dots+\Big(\sum_{k_1,\dots,k_{n-1}<k_n}{\cdots}\Big)\\
     &=: a^{(1)}(x,\xxi)+a^{(2)}(x,\xxi)+\dots+a^{(n)}(x,\xxi).
\end{align*} 
Then, due to the symmetry, it is enough to work only with $a^{(1)}$ and our actual goal is to show that if $s>\tau_{p,q}$
then
\begin{equation}\label{maingoal}
\big\Vert T_{[a^{(1)}]}\fff\big\Vert_{F_p^{s,q}}\lesssim \mathcal{A}_{N}^{(m)} \Vert f_1\Vert_{F_{p_1}^{s+m,q}}\prod_{j=2}^{n}{\Vert f_j\Vert_{h^{p_j}}}
\end{equation} for sufficiently large $N>0$ and  ${1}/{p}={1}/{p_1}+\dots+{1}/{p_n}$.

Observe that
\begin{equation*}
a^{(1)}(x,\xxi)=\sum_{k=0}^{\infty}{a(x,\xxi)\widehat{\phi_{k}}(\xi_1)\widehat{\Phi_{k}}(\xi_2)\cdots \widehat{\Phi_{k}}(\xi_n)}=:\sum_{k=0}^{\infty}{a_k(x,\xxi)}.
\end{equation*}
 Then each $a_k$ belongs to $\MM_n\mathcal{S}_{1,1}^{m}$ and for  $N\in\mathbb{N}_0$
\begin{equation}\label{symbolak}
\Vert a_k\Vert_{\MM_n\mathcal{S}_{1,1,N}^{m}}\lesssim \mathcal{A}_N^{(m)} \quad \text{unifomly in }~ k.
\end{equation}

Let $\{\widetilde{\phi_k}\}_{k\in\mathbb{N}_0}$ be a collection of Schwartz functions so that $\widehat{\widetilde{\phi_0}}(\xi):=\widehat{\phi_0}(\xi)+\widehat{\phi_1}(\xi)(=\widehat{\Phi_1}(\xi) )$ and 
$\widehat{\widetilde{\phi_k}}(\xi)=\widehat{\phi_k}(2\xi)+\widehat{\phi_k}(\xi)+\widehat{\phi_k}(2^{-1}\xi)$ for $k\geq 1$.
By using Fourier series expansion and the fact that $\widehat{\widetilde{\phi_k}}=1$ on $Supp(\widehat{\phi_k})$ and $\widehat{{\Phi_{k+1}}}=1$ on $Supp(\widehat{\Phi_k})$, one can write
\begin{equation*}
a_k(x,\xxi)=\sum_{\ll\in (\mathbb{Z}^d)^n}{c_k^{\ll}(x)\varphi_k^{l_1}(\xi_1)\vartheta_k^{l_2}(\xi_2)\cdots\vartheta_k^{l_n}(\xi_n)}
\end{equation*}
where 
\begin{equation*}
c_k^{\ll}(x):=\int_{(\mathbb{R}^d)^n}{a_k(x,2^k\eta_1,\dots,2^k\eta_n)e^{-2\pi i\langle \eta_1,l_1\rangle}\dots e^{-2\pi i\langle \eta_n,l_n\rangle}}d\eeta
\end{equation*}
\begin{equation*}
\varphi_k^{l_1}(\xi_1):=e^{2\pi i\langle l_1,2^{-k}\xi_1\rangle}\widehat{\widetilde{\phi_k}}(\xi_1), \quad \vartheta_k^{l_j}(\xi_j):=e^{2\pi i\langle l_j,2^{-k}\xi_j\rangle}\widehat{{\Phi_{k+1}}}(\xi_j), ~2\leq j\leq n.
\end{equation*}
It can be verified that for $l\in\mathbb{Z}^d$ and multi-index $\alpha$ one has
\begin{equation*}
Supp(\varphi_0^{l})\subset \{\xi\in\mathbb{R}^d: |\xi|\leq 2\}, \quad Supp(\varphi_k^l)\subset \{\xi\in\mathbb{R}^d: 2^{k-3}\leq |\xi|\leq 2^{k+1}\}~\text{ for }~k\geq 1,
\end{equation*}
\begin{equation*}
Supp(\vartheta_k^l)\subset \{\xi\in\mathbb{R}^d: |\xi|\leq 2^{k+1}\}~\text{ for }~k\geq 0,
\end{equation*}
\begin{equation*}
\big| \partial_{\xi}^{\alpha}\varphi_k^l(\xi)\big|,~ \big| \partial_{\xi}^{\alpha}\vartheta_k^l(\xi)\big| \lesssim 2^{-k|\alpha|} \quad \text{ for }~k\geq 0.
\end{equation*}
We rewrite $a_k(x,\xxi)$ as
\begin{equation*}
a_k(x,\xxi)=\sum_{\ll\in (\mathbb{Z}^d)^n}\sum_{u=0}^{\infty}{c_{k,u}^{\ll}(x)\varphi_k^{l_1}(\xi_1)\vartheta_k^{l_2}(\xi_2)\cdots\vartheta_k^{l_n}(\xi_n)}=: \sum_{\ll\in (\mathbb{Z}^d)^n}\sum_{u=0}^{\infty} A_{k,u}^{\ll}(x,\xxi)
\end{equation*}
where
\begin{equation*}
c_{k,0}^{\ll}:=\Phi_k\ast c_k^{\ll} \quad \text{ (low frequency part) }
\end{equation*}
\begin{equation*}
c_{k,u}^{\ll}:=\phi_{k+u}\ast c_k^{\ll},\quad u\geq 1 \quad \text{ (high frequency part) }.
\end{equation*}
Then
\begin{equation*}
T_{[a^{(1)}]}\fff=\sum_{\ll\in (\mathbb{Z}^d)^n}\sum_{k,u\in\mathbb{N}_0}{T_{[A_{k,u}^{\ll}]}\fff}.
\end{equation*}
and
\begin{equation*}
\big\Vert T_{[a^{(1)}]}\fff\big\Vert_{F_p^{s,q}}^{\min{(1,p,q)}}\leq \sum_{\ll\in (\mathbb{Z}^d)^n}{\Big\Vert \sum_{k,u\in\mathbb{N}_0}{T_{[A_{k,u}^{\ll}]}\fff}\Big\Vert_{F_p^{s,q}}^{\min{(1,p,q)}} }.
\end{equation*}
Therefore the proof of (\ref{maingoal}) can be deduced from the estimate that 
\begin{equation}\label{maindeduce}
\Big\Vert \sum_{k,u\in\mathbb{N}_0}{T_{[A_{k,u}^{\ll}]}\fff}\Big\Vert_{F_p^{s,q}}\lesssim \Big(\prod_{j=1}^{n}\frac{1}{(1+|l_j|)^{\mathcal{J}}}\Big) \mathcal{A}_N^{(m)}\Vert f_1\Vert_{F_{p_1}^{s+m,q}}\prod_{j=2}^{n}\Vert f_j\Vert_{h^{p_j}} 
\end{equation}
for sufficiently large $N>0$ and some $ \mathcal{J}   >d/\min{(1,p,q)}$.

\subsection{Pointwise estimate of $T_{[A_{k,u}^{\ll}]}\fff$}
Choose $N$ and $\sigma$ such that $N>s$, $N>d/\min{(1,p,q)}+d/\min{(p_1,\dots,p_n,q)}$, and 
$d/\min{(p_1,\dots,p_n,q)}<\sigma<N-d/\min{(1,p,q)}$. Let $\mathcal{J}:= N-\sigma ( > d/\min{(1,p,q)})$.
Then we will prove that
\begin{align}\label{taest}
&\big|T_{[A_{k,u}^{\ll}]}\fff(x) \big|\quad \lesssim \Big(\prod_{j=1}^{n}\frac{1}{\big(1+|l_j| \big)^{\mathcal{J}}} \Big)\mathcal{A}_N^{(m)}2^{km}2^{-uN} \mathfrak{M}_{\sigma,2^k}(f_1)_k(x)\Big(\prod_{j=2}^{n}{\mathfrak{M}_{\sigma,2^k}(f_j)^k(x)} \Big).
\end{align} 
We first see that
\begin{align}\label{taest1}
&\big|T_{[A_{k,u}^{\ll}]}\fff(x) \big|=\Big|c_{k,u}^{\ll}(x) \varphi_k^{l_1}(D)f_1(x)\prod_{j=2}^{n}\vartheta_k^{l_j}(D)f_j(x)\Big|.
\end{align}

Let $\phi_0^*:=\Phi_2$ and ${\phi_k^*}$ be Schwartz functions such that
\begin{equation*} 
 Supp(\widehat{{\phi_k^*}})\subset \big\{\xi: 2^{k-4}\leq |\xi|\leq 2^{k+2} \big\}, \quad \widehat{\phi_k^*}=1 \text{ on }~ Supp(\widehat{\widetilde{\phi_k}}), \quad \text{for }~k\geq 1.  
\end{equation*}
Setting
\begin{equation*}
(f_1)_{k}:={{\phi_k^*}}\ast f_1 \quad \text{ and } \quad (f_j)^{k}:={{\Phi_{k+2}}}\ast f_j, \quad 2\leq j\leq n,
\end{equation*} 
one obtains that
\begin{align}\label{varphiest}
\big|\varphi_k^{l_1}(D)f_1(x)\big|&= \big|\varphi_k^{l_1}(D)(f_1)_k(x)\big|\leq \int_{\mathbb{R}^d}{\big|\big(\varphi_k^{l_1}\big)^{\vee}(y)(f_1)_k(x-y) \big|}dy\nonumber\\
&\leq \mathfrak{M}_{\sigma,2^k}(f_1)_k(x)\int_{\mathbb{R}^d}{\big(1+2^k|y| \big)^{\sigma}\big|\widetilde{\phi_k}\big(y+2^{-k}l_1\big) \big|}dy\nonumber\\
& \lesssim \big( 1+|l_1|\big)^{\sigma}\mathfrak{M}_{\sigma,2^k}(f_1)_k(x),
\end{align}
and a similar analysis reveals that for each $2\leq j\leq n$
\begin{equation}\label{varthetaest}
\big| \vartheta_k^{l_j}(D)f_j(x)\big| \lesssim \big(1+|l_j|\big)^{\sigma}\mathfrak{M}_{\sigma,2^k}(f_j)^k(x).
\end{equation}

We now claim that
\begin{equation}\label{ckuest}
|c_{k,u}^{\ll}(x)|\lesssim \Big(\prod_{j=1}^{n}{\frac{1}{(1+|l_j|)^N}} \Big)\mathcal{A}_N^{(m)}2^{km}2^{-uN}, \quad \text{ uniformly in } ~ \ll.
\end{equation}
By applying integration by parts and (\ref{symbolak}),  one has
\begin{align*}
\big|\partial_x^{\beta}c_k^{\ll}(x)\big|&\lesssim \Big(\prod_{j=1}^{n}{\frac{1}{(1+|l_j|)^N}} \Big) \sum_{|\alpha_1|,\dots,|\alpha_n|\leq N}{\int_{(\mathbb{R}^d)^n}{\big| \partial_{\eeta}^{\aa}\partial_{x}^{\beta}a_k(x,2^k\eta_1,\dots,2^k\eta_n)\big|2^{k|\aa|}}d\eeta}\\
&\lesssim_N \Big(\prod_{j=1}^{n}{\frac{1}{(1+|l_j|)^N}} \Big) \mathcal{A}_N^{(m)}2^{k(m+|\beta|)}
\end{align*}
where the second follows from the fact that the domain of the integral is actually $\big\{\eeta\in (\mathbb{R}^d)^n: |\eta_j|\leq 2 \text{ for } 1\leq j\leq n \big\}$.
This yields that
\begin{equation*}
\big|c_{k,0}^{\ll}(x)\big|\lesssim \big\Vert c_k^{\ll}\big\Vert_{L^{\infty}}\lesssim \Big(\prod_{j=1}^{n}{\frac{1}{(1+|l_j|)^N}} \Big) \mathcal{A}_N^{(m)}2^{km},
\end{equation*}
and for $u\geq 1$ one has
\begin{equation*}
|c_{k,u}^{\ll}(x)|\lesssim \sum_{\beta:|\beta|=N}{\big\Vert \partial_x^{\beta}c_k^{\ll}\big\Vert_{L^{\infty}}\int_{\mathbb{R}^d}{|y|^N|\phi_{k+u}(y)|}dy}\lesssim  \Big(\prod_{j=1}^{n}{\frac{1}{(1+|l_j|)^N}} \Big) \mathcal{A}_N^{(m)} 2^{km}2^{-uN}
\end{equation*} by using the vanishing moment property of $\phi_{k+u}$.
This proves (\ref{ckuest}).

 Finally,  (\ref{taest1}), (\ref{varphiest}),  (\ref{varthetaest}), and (\ref{ckuest}) establish (\ref{taest}).

\subsection{Proof of (\ref{maindeduce})}

We observe that
\begin{equation}\label{supportt}
Supp(\widehat{T_{[A_{k,u}^{\ll}]}\fff})\subset \big\{ \xi\in\mathbb{R}^d :|\xi|\leq 2^{k+u}+n2^k\big\}
\end{equation} and this yields, with the support condition of $\widehat{\phi_h}$, that
for $h\in\mathbb{N}_0$
\begin{equation*}
\phi_h\ast \Big(\sum_{k,u\in\mathbb{N}_0}{T_{[A_{k,u}^{\ll}]}}\fff\Big)=\sum_{\substack{u,k\in\mathbb{N}_0\\ k+u+3+\lfloor \log_2{n}\rfloor\geq h}}{\phi_h\ast T_{[A_{k,u}^{\ll}]}\fff }.
\end{equation*}
By assuming $A_{k,u}^{\ll}=0$ for $k<0$ and applying a change of variables, the last expression is
\begin{align*}
&\sum_{u=0}^{\infty}\sum_{k=h-u-3-\lfloor \log_2{n}\rfloor}^{\infty}{\phi_h\ast T_{[A_{k,u}^{\ll}]}\fff  }\\
&=\sum_{u,v\in\mathbb{N}_0}\phi_h\ast T_{[A_{v+h-u-3-\lfloor \log_2{n}\rfloor,u}^{\ll}]}\fff = \phi_h\ast\Big(\sum_{u,v\in\mathbb{N}_0} T_{[A_{v+h-u-3-\lfloor \log_2{n}\rfloor,u}^{\ll}]}\fff\Big).
\end{align*}
That is, for $h\in \mathbb{N}_0$
\begin{equation}\label{transform}
\phi_h\ast \Big(\sum_{k,u\in\mathbb{N}_0}{T_{[A_{k,u}^{\ll}]}}\fff\Big)=\phi_h\ast\Big(\sum_{u,v\in\mathbb{N}_0} T_{[A_{v+h-u-3-\lfloor \log_2{n}\rfloor,u}^{\ll}]}\fff\Big).
\end{equation}

Moreover, a proper use of Calder\'on's reproducing formula proves that
\begin{equation}\label{localhardy1}
\big\Vert \sup_{k\in\mathbb{N}_0}{|(f_j)^k|}\big\Vert_{L^p}\lesssim \Vert f_j\Vert_{h^p}, \quad 0<p\leq \infty, 
\end{equation}
\begin{equation}\label{fpsq}
\big\Vert \big\{ 2^{sk}(f_j)_k\big\}_{k\in\mathbb{N}_0}\big\Vert_{L^p(l^q)} \lesssim \Vert f_j\Vert_{F_p^{s,q}}, \quad p<\infty ~\text{ or }~ p=q=\infty,
\end{equation}
\begin{equation}\label{fisq}
\sup_{P\in\mathcal{D}, l(P)<1}\Big(\frac{1}{|P|}\int_{P}{\sum_{k=-\log_2{l(P)}}^{\infty}{2^{skq}\big| (f_j)_k(x)\big|^q}}dx \Big)^{1/q} \lesssim \Vert f_j\Vert_{F_{\infty}^{s,q}}, \quad 0<q<\infty.
\end{equation}

\subsubsection{The case $0<p<\infty$ or $p=q=\infty$}
From (\ref{transform}) one has
\begin{equation*}
\Big\Vert \Big\{2^{sh} \phi_h\ast \Big(\sum_{k,u\in\mathbb{N}_0}{T_{[A_{k,u}^{\ll}]}\fff}\Big)\Big\}_{h\in\mathbb{N}_0}\Big\Vert_{L^p(l^q)}^{\min{(1,p,q)}}\leq \sum_{u,v\in\mathbb{N}_0}{\Big\Vert \Big\{ 2^{sh}\phi_h\ast T_{[A_{v+h-u-3-\lfloor \log_2{n}\rfloor,u}^{\ll}]}\fff \Big\}_{h\in\mathbb{N}_0}  \Big\Vert_{L^p(l^q)}^{\min{(1,p,q)}}}.
\end{equation*}
It follows from  (\ref{supportt}) that the Fourier transform of $T_{[A_{v+h-u-3-\lfloor \log_2{n}\rfloor,u}^{\ll}]}\fff$ is supported on $ \big\{|\xi|\leq 2^{v+h}\big\}$.
We choose $t>0$ such that $s>t-d/2>\tau_{p,q}$ and apply Lemma \ref{propo} (1)  to obtain
\begin{align*}
&\Big\Vert \Big\{ 2^{sh}\phi_h\ast T_{[A_{v+h-u-3-\lfloor \log_2{n}\rfloor,u}^{\ll}]}\fff \Big\}_{h\in\mathbb{N}_0}  \Big\Vert_{L^p(l^q)}\nonumber\\
 &\lesssim  \sup_{l\in\mathbb{N}}{\big\Vert \widehat{\phi_l}(2^{v+l})\big\Vert_{L_t^2}} \Big\Vert \Big\{ 2^{sh}T_{[A_{v+h-u-3-\lfloor \log_2{n}\rfloor,u}^{\ll}]}\fff\Big\}_{h\in\mathbb{N}_0}\Big\Vert_{L^p(l^q)}\nonumber\\
&\approx 2^{v(t-d/2)}\Big\Vert \Big\{ 2^{sh}T_{[A_{v+h-u-3-\lfloor \log_2{n}\rfloor,u}^{\ll}]}\fff\Big\}_{h\in\mathbb{N}_0}\Big\Vert_{L^p(l^q)}\nonumber\\
&\lesssim 2^{-v(s-t+d/2)}2^{su}\Big\Vert \Big\{ 2^{sk}T_{[A_{k,u}^{\ll}]}\fff\Big\}_{k\in\mathbb{N}_0}\Big\Vert_{L^{p}(l^q)}\nonumber\\
&\lesssim \Big(\prod_{j=1}^{n}{\frac{1}{(1+|l_j|)^{\mathcal{J}}}} \Big)\mathcal{A}_N^{(m)}2^{-v(s-t+d/2)}2^{-u(N-s)}\\
&\relphantom{=}\times \Big\Vert \Big\{ 2^{k(s+m)}\mathfrak{M}_{\sigma,2^k}{(f_1)_k}\prod_{j=2}^{n}{\mathfrak{M}_{\sigma,2^k}{(f_j)^k}}\Big\}_{k\in\mathbb{N}_0}\Big\Vert_{L^p(l^q)}\nonumber
\end{align*} where we applied a change of variables and (\ref{taest}) in the last two inequalities. 
Since $s-t+d/2>0$ and  $N-s>0$, the left hand side of (\ref{maindeduce}) is majored by a constant multiple of
\begin{equation*}
 \Big(\prod_{j=1}^{n}{\frac{1}{(1+|l_j|)^{\mathcal{J}}}} \Big)\mathcal{A}_N^{(m)} \Big\Vert \Big\{2^{k(s+m)}  \mathfrak{M}_{\sigma,2^k}(f_1)_k\prod_{j=2}^{n}{\mathfrak{M}_{\sigma,2^k}(f_j)^k} \Big\}_{k\in\mathbb{N}_0}\Big\Vert_{L^p(l^q)}.
\end{equation*} 
Moreover, using (\ref{infmax2}),
\begin{align*}
&\Big\Vert \Big\{2^{k(s+m)}  \mathfrak{M}_{\sigma,2^k}(f_1)_k\prod_{j=2}^{n}{\mathfrak{M}_{\sigma,2^k}(f_j)^k} \Big\}_{k\in\mathbb{N}_0}\Big\Vert_{L^p(l^q)}\nonumber\\
&\lesssim \Big\Vert \Big\{ \sum_{Q\in\mathcal{D}_k}2^{k(s+m)}  \Big(\inf_{y\in Q}{\mathfrak{M}_{\sigma,2^k}(f_1)_k(y)}\Big)\Big[\prod_{j=2}^{n}{\Big(\inf_{y\in Q}{\mathfrak{M}_{\sigma,2^k}(f_j)^k}(y)\Big)}\Big]\chi_Q \Big\}_{k\in\mathbb{N}_0}\Big\Vert_{L^p(l^q)}.
\end{align*}
Now let $S_Q:=S_Q^{\gamma,q}(\{\mathfrak{M}_{\sigma,2^k}(f_1)_k\}_{k\in\mathbb{N}_0})$ and apply (\ref{inversechi}) and (\ref{hlmax}) for $0<r<\min{(1,p)}$ to show that the last expression is 
\begin{align*}
&\lesssim \Big\Vert \Big\{ \sum_{Q\in\mathcal{D}_k}2^{k(s+m)}  \big( \inf_{y\in Q}{\mathfrak{M}_{\sigma,2^k}(f_1)_k(y)}\Big) \Big[\prod_{j=2}^{n}{\Big( \inf_{y\in Q}{\mathfrak{M}_{\sigma,2^k}(f_j)^k}(y)\Big)}\Big]\chi_{S_Q} \Big\}_{k\in\mathbb{N}_0}\Big\Vert_{L^p(l^q)}\\
&\leq \Big\Vert \Big\{\sum_{Q\in\mathcal{D}_k}{2^{(s+m)k}\inf_{y\in Q}{\mathfrak{M}_{\sigma,2^k}(f_1)_k(y)}\chi_{S_Q}} \Big\}_{k\in\mathbb{N}}\Big\Vert_{L^{p_1}(l^q)}\prod_{j=2}^{n}\big\Vert \big\{ \mathfrak{M}_{\sigma,2^k}(f_j)^k\big\}_{k\in\mathbb{N}_0}\big\Vert_{L^{p_j}(l^{\infty})}\\
&\lesssim \Vert f_1\Vert_{F_{p_1}^{s+m,q}}\prod_{j=2}^{n}{\Vert f_j\Vert_{h^{p_j}}}
\end{align*}
where H\"older's inequality, Corollary \ref{cor1},  Corollary \ref{cor2} (with $\mu=0$), and (\ref{localhardy1})-(\ref{fisq}) are applied.

Combining all together the proof of (\ref{maindeduce}) ends for $0<p<\infty$ or $p=q=\infty$.

\subsubsection{The case $p=\infty$ and $0<q<\infty$}
Suppose $\sigma>d/q$.
First of all, by using (\ref{maindeduce}) for the case $p=q=\infty$ and the embedding $F_{\infty}^{s+m,q}\hookrightarrow F_{\infty}^{s+m,\infty}$ one has
\begin{align*}
 \Big\Vert \phi_0\ast \Big( \sum_{k,u\in\mathbb{N}_0}{T_{[A_{k,u}^{\ll}]}\fff}\Big)\Big\Vert_{L^{\infty}}&\leq \Big\Vert \sum_{k,u \in\mathbb{N}_0}{T_{[A_{k,u}^{\ll}]}\fff}\Big\Vert_{F_{\infty}^{s,\infty}}\\
 &\lesssim \Big(\prod_{j=1}^{n}{\frac{1}{(1+|l_j|)^{\mathcal{J}}}} \Big)\mathcal{A}_N^{(m)}\Vert f_1\Vert_{F_{\infty}^{s+m,q}}\prod_{j=2}^{n}{\Vert f_j\Vert_{L^{\infty}}}.
\end{align*}

Now we fix a dyadic cube $P\in\mathcal{D}$ with $l(P)<1$.
Then it follows from (\ref{transform}) that
\begin{align}\label{suptermest}
&\Big(\frac{1}{|P|}\int_P{\sum_{h=-\log_2{l(P)}}^{\infty}{2^{shq}\Big|\phi_h\ast \Big( \sum_{k,u\in\mathbb{N}_0}{T_{[A_{k,u}^{\ll}]}\fff(x)}\Big) \Big|^q}}dx \Big)^{1/q}\nonumber\\
&\leq \Big[ \sum_{u,v\in\mathbb{N}_0}  {\Big(\frac{1}{|P|}\int_P{\sum_{h=-\log_2{l(P)}}^{\infty}{2^{shq} \Big|\phi_h\ast T_{[A_{v+h-u-3-\lfloor \log_2{n}\rfloor,u}^{\ll}]}\fff(x)\Big|^q    }}dx \Big)^{\min{(1,q)}/q}}\Big]^{1/\min{(1,q)}}.    
\end{align}
We choose $t>0$ such that $s>t-d/2>\tau_q$ and apply Lemma \ref{propo} (2) with $\mu=1$. Then
\begin{align*}
&{\Big(\frac{1}{|P|}\int_P{\sum_{h=-\log_2{l(P)}}^{\infty}{2^{shq} \Big|\phi_h\ast T_{[A_{v+h-u-3-\lfloor \log_2{n}\rfloor,u}^{\ll}]}\fff(x)\Big|^q    }}dx \Big)^{1/q}}\\
&\lesssim \sup_{l\in\mathbb{N}_0}{\big\Vert \widehat{\phi_l}(2^{v+l}\cdot )\big\Vert_{L_t^2}}\sup_{R\in\mathcal{D}, l(R)<1}{\Big( \frac{1}{|R|}\int_P{\sum_{h=-\log_2{l(R)}}^{\infty}{2^{shq}\big| T_{[A_{v+h-u-3-\lfloor \log_2{n}\rfloor,u}^{\ll}]}\fff(x)\big|^q     }}dx\Big)^{1/q}}\\
&\lesssim 2^{-v(s-t+d/2)}2^{su}\sup_{R\in\mathcal{D},l(R)<1}{\Big(\frac{1}{|R|}\int_R{\sum_{k=v-u-3-\lfloor \log_2{n}\rfloor-\log_2{l(R)}}^{\infty}{2^{skq}\big|T_{[A_{k,u}^{\ll}]}\fff(x) \big|^q}}dx \Big)^{1/q}}.
\end{align*}
We deal with only the case $v-u-3-\lfloor \log_2{n}\rfloor\leq -1$ since the other case follows in a similar and simpler way.
The supremum in the last expression is less than a constant times the sum of
\begin{equation}\label{fpart}
\sup_{R\in\mathcal{D},l(R)<1}{\Big(\frac{1}{|R|}\int_R \sum_{k=-\log_2{l(R)}}^{\infty}{2^{skq}\big|  T_{[A_{k,u}^{\ll}]}\fff(x)  \big|^q} dx\Big)^{1/q}}
\end{equation}
\begin{equation}\label{spart}
\sup_{R\in\mathcal{D},l(R)<1}{\Big(\frac{1}{|R|}\int_R \sum_{k=v-u-3-\lfloor \log_2{n}\rfloor-\log_2{l(R)}}^{-\log_2{l(R)}-1}{2^{skq}\big|  T_{[A_{k,u}^{\ll}]}\fff(x)  \big|^q} dx\Big)^{1/q} }.
\end{equation}

We see that
\begin{equation*}
(\ref{spart})\lesssim (u+1)  \Big\Vert \Big\{ 2^{sk}T_{[A_{k,u}^{\ll}]}\fff\Big\}_{k\in\mathbb{N}_0}\Big\Vert_{L^{\infty}(l^{\infty})} 
\end{equation*}
and by using (\ref{taest}), (\ref{localhardy1}), (\ref{fpsq}), and the embedding $F_{\infty}^{s+m,q}\hookrightarrow F_{\infty}^{s+m,\infty}$,
\begin{align*}
&\Big\Vert \Big\{ 2^{sk}T_{[A_{k,u}^{\ll}]}\fff\Big\}_{k\in\mathbb{N}_0}\Big\Vert_{L^{\infty}(l^{\infty})} \\
&\lesssim \Big(\prod_{j=1}^{n}\frac{1}{\big(1+|l_j|\big)^{\mathcal{J}}} \Big)\mathcal{A}_N^{(m)}2^{-uN}\big\Vert \big\{2^{(s+m)k}\mathfrak{M}_{\sigma,2^k}(f_1)_k \big\}_{k\in\mathbb{N}_0}\big\Vert_{L^{\infty}(l^{\infty})}\Big( \prod_{j=2}^{n}{\big\Vert f_j }\big\Vert_{L^{\infty}}\Big)\\
&\lesssim \Big(\prod_{j=1}^{n}\frac{1}{\big(1+|l_j|\big)^{\mathcal{J}}} \Big)\mathcal{A}_N^{(m)}2^{-uN}\Vert f_1\Vert_{F_{\infty}^{s+m,q}}\Big( \prod_{j=2}^{n}{\Vert f_j\Vert_{L^{\infty}}}\Big).
\end{align*} 
 This proves that the term corresponding to (\ref{spart}) in (\ref{suptermest}) is dominated by a constant times
\begin{equation*}
\Big(\prod_{j=1}^{n}\frac{1}{\big(1+|l_j|\big)^{\mathcal{J}}} \Big)\mathcal{A}_N^{(m)}\Vert f_1\Vert_{F_{\infty}^{s+m,q}}\Big( \prod_{j=2}^{n}{\Vert f_j\Vert_{L^{\infty}}}\Big)
\end{equation*} 
because 
\begin{equation*}
\Big( \sum_{u,v\in\mathbb{N}_0}{2^{-v(s-t+d/2)\min{(1,q)}}2^{-u(N-s)(\min{(1,q)})}(u+1)^{\min{(1,q)}}}\Big)^{1/\min{(1,q)}}\lesssim 1.
\end{equation*}

Similarly, (\ref{taest}) yields that for $N>s$
\begin{align*}
(\ref{fpart})&\lesssim \Big(\prod_{j=1}^{n}\frac{1}{\big(1+|l_j|\big)^{\mathcal{J}}} \Big)\mathcal{A}_N^{(m)}2^{-uN}\Big( \prod_{j=2}^{n}{\big\Vert f_j\big\Vert_{L^{\infty}}}\Big)\\
&\relphantom{=}\times\sup_{P\in\mathcal{D},l(P)<1}{\Big(\frac{1}{|P|}\int_P{\sum_{k=-\log_2{l(P)}}^{\infty}{2^{k(s+m)q}\big(\mathfrak{M}_{\sigma,2^k}(f_1)_k(x) \big)^q}}dx \Big)^{1/q}}\\
&\lesssim \Big(\prod_{j=1}^{n}\frac{1}{\big(1+|l_j|\big)^{\mathcal{J}}} \Big)\mathcal{A}_N^{(m)}2^{-uN}\Vert f_1\Vert_{F_{\infty}^{s+m,q}}\Big( \prod_{j=2}^{n}{\Vert f_j\Vert_{L^{\infty}}}\Big)
\end{align*} where Lemma \ref{maximal3} (2) and (\ref{fisq}) are applied in the last inequality.
This implies that the term corresponding to (\ref{fpart}) in (\ref{suptermest}) is also bounded by a constant times
\begin{equation*}
\Big(\prod_{j=1}^{n}\frac{1}{\big(1+|l_j|\big)^{\mathcal{J}}} \Big)\mathcal{A}_N^{(m)}\Vert f_1\Vert_{F_{\infty}^{s+m,q}}\Big( \prod_{j=2}^{n}{\Vert f_j\Vert_{L^{\infty}}}\Big).
\end{equation*}

This completes the proof of (\ref{maindeduce}) for $p=\infty$ and $0<q<\infty$.


\appendix
\section{The proof of (\ref{bmoh1})} \label{appendixa}
Suppose $0<t<1$ and $\sigma>d/t>d$.
For $Q\in \mathcal{D}$ let $S_Q$ be the subset of $Q$ for $BMO$ norm equivalence of $f$ in Corollary \ref{maincor}. For $k\in\mathbb{Z}$ let $\widetilde{\phi_k}:= \phi_{k-1}+\phi_k+\phi_{k+1}$ so that $\widehat{\widetilde{\phi_k}}(\xi)\widehat{\phi_k}(\xi)=\widehat{\phi_k}(\xi)$. Then by applying (\ref{infmax2}), (\ref{inversechi}), Lemma \ref{maximal2} (1), Corollary \ref{maincor}, and Lemma \ref{character2} (1), one obtains
\begin{align*}
\big|\langle f,g\rangle\big|& \leq \int_{\mathbb{R}^d}{\sum_{k\in\mathbb{Z}}{\big|\phi_k\ast f(x)\big|\big|\widetilde{\phi_k}\ast g(x)\big|}}dx=\int_{\mathbb{R}^d}{\sum_{k\in\mathbb{Z}}{\sum_{Q\in\mathcal{D}_k}\big|\phi_k\ast f(x)\big|\big|\widetilde{\phi_k}\ast g(x)\big|\chi_Q(x)}}dx\\
 &\lesssim \int_{\mathbb{R}^d}{\sum_{k\in\mathbb{Z}}{\sum_{Q\in\mathcal{D}_k}\inf_{y\in Q}\big(\mathfrak{M}_{\sigma,2^k}^t(\phi_k\ast f)(y)\big)\inf_{y\in Q}\big(\mathfrak{M}_{\sigma,2^k}^t(\widetilde{\phi_k}\ast g)(y)\big)\chi_Q(x)}}dx\\
 &\lesssim  \int_{\mathbb{R}^d}{\sum_{k\in\mathbb{Z}}{\sum_{Q\in\mathcal{D}_k}\inf_{y\in Q}\big(\mathfrak{M}_{\sigma,2^k}^t(\phi_k\ast f)(y)\big)\inf_{y\in Q}\big(\mathfrak{M}_{\sigma,2^k}^t(\widetilde{\phi_k}\ast g)(y)\big)\chi_{S_Q}(x)}}dx\\
 &\leq  \Big\Vert \Big\{ \sum_{Q\in\mathcal{D}_k}{\inf_{y\in Q}{\big(\mathfrak{M}_{\sigma,2^k}^t\big(\phi_k\ast f\big)(y) \big)}\chi_{S_Q}}\Big\}_{k\in\mathbb{Z}}\Big\Vert_{L^{\infty}(l^2)}\\
 &\relphantom{=}\times  \Big\Vert \Big\{ \sum_{Q\in\mathcal{D}_k}{\inf_{y\in Q}{\big(\mathfrak{M}_{\sigma,2^k}^t\big(\widetilde{\phi_k}\ast g\big)(y) \big)}\chi_{Q}}\Big\}_{k\in\mathbb{Z}}\Big\Vert_{L^{1}(l^2)}\\
 &\approx \Vert f\Vert_{BMO}\Vert g\Vert_{H^1}
\end{align*}
where we used the fact that $\Vert g\Vert_{\dot{F}_1^{0,2}}\approx \Vert g\Vert_{H^1}$.

\section{Proof of Lemma \ref{comparelemma}, \ref{mcomposition}, and \ref{comparelemma2}}\label{appendixb}

\subsection{Proof of Lemma \ref{comparelemma}}
Since the case $t=s$ is trivial, we only consider the case $t<s$.
Let  $\Psi_0\in S$ satisfy 
\begin{equation*}
Supp(\widehat{\Psi_0})\subset \big\{\xi:|\xi| \leq 2^2A \big\} \qquad \text{ and }\qquad \widehat{\Psi_0}(\xi)=1\quad\text{for }~ |\xi|\leq 2A.
\end{equation*}
Then we note that $f=\Psi_k\ast f$.

First, assume $s=\infty$ and $0<t<\infty$.
If $1<t<\infty$, then it follows from H\"older's inequality that
\begin{align*}
\frac{|f(x-y)|}{(1+2^k|y|)^{\sigma}}&\leq \int_{\mathbb{R}^d}{|f(x-z)|\frac{|\Psi_k(z-y)|}{(1+2^k|y|)^{\sigma}}}dz\\
 &\leq \int_{\mathbb{R}^d}{\frac{|f(x-z)|}{(1+2^k|z|)^{\sigma}}|\Psi_k(z-y)|(1+2^k|z-y|)^{\sigma}}dz\\
 &\leq \mathfrak{M}_{\sigma,2^k}^tf(x) 2^{-kd/t}\Big(\int_{\mathbb{R}^d}{\big(|\Psi_k(z)|(1+2^k|z|)^{\sigma} \big)^{t'}}dz \Big)^{1/t'}\\
 &\lesssim  \mathfrak{M}_{\sigma,2^k}^tf(x).
\end{align*}
If $0<t\leq 1$ then we apply Nikolskii's inequality  to obtain
\begin{equation*}
|f(x-y)|\lesssim 2^{kd(1/t-1)}\Big(\int_{\mathbb{R}^d}{|f(x-z)|^t|\Psi_k(z-y)|^t}dz \Big)^{1/t}
\end{equation*} and thus
\begin{align*}
\frac{|f(x-y)|}{(1+2^k|y|)^{\sigma}}&\lesssim 2^{kd(1/t-1)}\Big( \int_{\mathbb{R}^d}{\frac{|f(x-z)|^t}{(1+2^k|z|)^{\sigma t}}|\Psi_k(z-y)|^t(1+2^k|z-y|)^{\sigma t}}dz\Big)^{1/t}\\
 &\lesssim \mathfrak{M}_{\sigma,2^k}^tf(x). 
\end{align*}
This proves 
\begin{equation}\label{inftycase}
\mathfrak{M}_{\sigma,2^k}f(x)\lesssim \mathfrak{M}_{\sigma,2^k}^{t}f(x).
\end{equation}

Now assume $0<t<s<\infty$.
Then one has
\begin{align*}
\mathfrak{M}_{\sigma,2^k}^sf(x)&=2^{kd/s}\Big(\int_{\mathbb{R}^d}{\Big(\frac{|f(x-y)|}{(1+2^k|y|)^{\sigma}} \Big)^s}dy \Big)^{1/s}\\
 &\leq \big( \mathfrak{M}_{\sigma,2^k}f(x)\big)^{1-t/s}2^{kd/s}\Big(\int_{\mathbb{R}^d}{\Big(\frac{|f(x-y)|}{(1+2^k|y|)^{\sigma}} \Big)^t}dy \Big)^{1/s}\\
 &\lesssim \big( \mathfrak{M}_{\sigma,2^k}^{t}f(x)\big)^{1-t/s}\big( \mathfrak{M}_{\sigma,2^k}^{t}f(x)\big)^{t/s}=\mathfrak{M}_{\sigma,2^k}^{t}f(x).
\end{align*} by applying (\ref{inftycase}).

\subsection{Proof of Lemma \ref{mcomposition}}

We only care about the case $0<t<s<\infty$ as the other cases can be done similarly.
By applying Minkowski's inequality with $s/t>1$, one has
\begin{align*}
\mathfrak{M}_{\sigma,2^k}^{s}\mathfrak{M}_{\sigma,2^k}^tf(x)&=2^{kd/s}2^{kd/t} \Big( \int_{\mathbb{R}^d}{\Big( \int_{\mathbb{R}^d}{\frac{|f(x-z)|^t}{(1+2^k|y-z|)^{\sigma t}(1+2^k|y|)^{\sigma t}}}dz\Big)^{s/t}}dy\Big)^{1/s}\\
&\leq 2^{kd/s} 2^{kd/t} \Big(\int_{\mathbb{R}^d}{|f(x-z)|^t\Big( \int_{\mathbb{R}^d}{\frac{1}{(1+2^k|y-z|)^{\sigma s}(1+2^k|y|)^{\sigma s}}}dy\Big)^{t/s}}dz \Big)^{1/t}
\end{align*}
and a standard computation (see \cite[Appendix B]{Gr}) yields that
\begin{equation*}
 \int_{\mathbb{R}^d}{\frac{1}{(1+2^k|y-z|)^{\sigma s}(1+2^k|y|)^{\sigma s}}}dy\lesssim \frac{2^{-kd}}{(1+2^k|z|)^{\sigma s}}.
\end{equation*}
Therefore,
\begin{equation*}
\mathfrak{M}_{\sigma,2^k}^{s}\mathfrak{M}_{\sigma,2^k}^tf(x)\lesssim \mathfrak{M}_{\sigma,2^k}^tf(x). \qedhere
\end{equation*}

\subsection{Proof of Lemma \ref{comparelemma2}}

Suppose $0<\epsilon<\sigma-d/t$.
Let 
\begin{equation*}
E_0 :=\big\{y\in\mathbb{R}^d:|y|\leq 2^{-k} \big\} ~\text{ and }~ E_j :=\big\{y\in\mathbb{R}^d:2^{-k+j-1}<|y|\leq 2^{-k+j} \big\}, ~j\geq 1.
\end{equation*}
Then one has
\begin{equation*}
\int_{\mathbb{R}^d}{\frac{|f(x-y)|^t}{(1+2^k|y|)^{\sigma t}}}dy\lesssim \sum_{j=0}^{\infty}{2^{-j\sigma t}\int_{E_j}{|f(x-y)|^t}dy}\leq 2^{-kd}\big(\mathcal{M}_t^{k,\epsilon}f(x) \big)^t\sum_{j=0}^{\infty}{2^{-jt(\sigma-d/t-\epsilon )}},
\end{equation*} which concludes the proof since $\sigma-d/t-\epsilon >0$.


\begin{thebibliography}{99}






\bibitem{Be_Na_To}
\'A. B\'enyi, A.R. Nahmod, and R. H. Torres, \emph{Sobolev space estimates and symbolic calculus for bilinear pseudodifferential operators}, J. Geom. Anal. \textbf{16}
(2006) 431-453.

\bibitem{Be_To}
\'A. B\'enyi  and R. H. Torres, \emph{Symbolic calculus and the transposes of bilinear pseudodifferential operators}, Comm. Partial Differ. Equ. \textbf{28}
(2003) 1161-1181.



\bibitem{Bo_Li}
J. Bourgain and D. Li,  \emph{On an endpoint Kato-Ponce inequality}, Diff. Integr. Equ. \textbf{27}
(2014) 1037-1072.

\bibitem{Br_Na}
J. Brummer and V. Naibo,  \emph{Bilinear operators with homogeneous symbols, smooth molecules, and Kato-Ponce inequalities}, Proc. Amer. Math. Soc., \textbf{146}
(2018) 1217-1230.


\bibitem{Car}
L. Carleson,  \emph{Two remarks on $H^1$ and $B.M.O.$}, Advances in Math. \textbf{22}
(1976) 269-277.

\bibitem{Co_Me1}
R. R. Coifman and Y. Meyer,  \emph{On commutators of singular integrals and bilinear singular integrals}, Trans. Amer. Math. Soc. \textbf{212}
(1975) 315-331.


\bibitem{Co_Me2}
R. R. Coifman and Y. Meyer,  \emph{Au del\`a des op\'erateurs pseudo-diff\'erentiels}, Ast\'erisque \textbf{57}
(1978) 1-185.


\bibitem{Co_Me3}
R. R. Coifman and Y. Meyer,  \emph{Nonlinear harmonic analysis, operator theory and PDE. In Beijing Lectures in Harmonic Analysis (Beijing, 1984)}, 3-45. Ann. of Math. Stud. \textbf{112}, Princeton Univ. Press, Princeton, NJ, 1986.


\bibitem{Co_Me4}
R. R. Coifman and Y. Meyer,  \emph{Wavelets: Calder\'on-Zygmund and Multilinear Operators}, Cambridge Stud. Adv. Math., \textbf{48}, Cambridge University Press, Cambridge 1997.

\bibitem{Fe}
C. Fefferman,  \emph{Characterizations of bounded mean oscillation}, Bull. Amer. Math. Soc. \textbf{77} (1971) 587-588.


\bibitem{Fe_St}
C. Fefferman and E. M. Stein,  \emph{Some maximal inequalities}, Amer. J. Math. \textbf{93} (1971) 107-115.


\bibitem{Fr_Ja2}
M. Frazier and B. Jawerth,  \emph{A discrete transform and decomposition of distribution spaces}, J. Func. Anal. \textbf{93}
(1990) 34-170.


\bibitem{Fu_Tom1}
M. Fugita and N. Tomita,  \emph{Weighted norm inequalities for multilinear Fourier multipliers}, Trans. Amer. Math. Soc. \textbf{364} (2012) 6335-6353.


\bibitem{Fu_Tom2}
M. Fugita and N. Tomita,  \emph{Weighted norm inequalities for multilinear Fourier multipliers with critical Besov regularity}, Proc. Amer. Math. Soc. \textbf{146} (2017) 555-569.



\bibitem{Go}
D. Goldberg,  \emph{A local version of real Hardy spaces}, Duke Math. J. \textbf{46}
(1979) 27-42.




\bibitem{Gr}
L. Grafakos,  \emph{Modern Fourier Analysis}, Third edition, Springer (2014)


\bibitem{Gr_Oh}
L. Grafakos and S. Oh,  \emph{The Kato-Ponce inequality}, Comm. Partial Differ. Equ. \textbf{39} (2014) 1128-1157.


\bibitem{Gr_Mi_Tom}
L. Grafakos, A. Miyachi, and N. Tomita,  \emph{On multilinear Fourier multipliers of limited smoothness}, Can. J. Math. \textbf{65} (2013) 299-330.

\bibitem{Gr_Mi_Ng_Tom}
L. Grafakos, A. Miyachi, H.V. Nguyen, and N. Tomita,  \emph{Multilinear Fourier multipliers with minimal Sobolev regularity, II}, J. Math. Soc. Japan \textbf{69} (2017) 529-562.

\bibitem{Gr_Ng}
L. Grafakos and H.V. Nguyen,  \emph{Multilinear Fourier multipliers with minimal Sobolev regularity, I}, Colloquium Math. \textbf{144} (2016) 1-30.


\bibitem{Gr_Si}
L. Grafakos and Z. Si,  \emph{The H\"ormander multiplier theorem for multilinear operators}, J. Reine Angew. Math. \textbf{668} (2012) 133-147.

\bibitem{Gr_To}
L. Grafakos and R.H. Torres,  \emph{Multilinear Calder\'on-Zygmund theory}, Adv. Math. \textbf{165} (2002) 124-164.



\bibitem{Ka_Po}
T. Kato and G. Ponce, \emph{Commutator estimates and the Euler and Navier-Stokes equations}, Comm. Pure Appl. Math. \textbf{41} (1988) 891-907.


\bibitem{Ko_To}
K. Koezuka and N. Tomita, \emph{Bilinear Pseudo-differential Operators with Symbols in $BS_{1,1}^{m}$ on Triebel-Lizorkin Spaces}, J. Fourier Anal. Appl. \textbf{24} (2018) 309-319.


\bibitem{La_Th1}
M. T. Lacey and C.M. Thiele, \emph{$L^p$ bounds for the bilinear Hilbert transform, $2<p<\infty$}, Ann. Math. \textbf{146} (1997) 693-724.

\bibitem{La_Th2}
M. T. Lacey and C.M. Thiele, \emph{On Calder\'on's conjecture}, Ann. Math. \textbf{149} (1999) 475-496.

\bibitem{Mic_Ru_St}
N. Michalowski, D. Rule, and W. Staubach, \emph{Multilinear pseudodifferential operators beyond Calder\'on-Zygmund theory}, J. Math. Anal. Appl.  \textbf{414} (2014) 149-165.



\bibitem{Mi_Tom}
A. Miyachi and N. Tomita, \emph{Minimal smoothness conditions for bilinear Fourier multipliers}, Rev. Mat. Iberoamericana  \textbf{29} (2013) 495-530.


\bibitem{Mi_Tom2}
A. Miyachi and N. Tomita, \emph{Calder\'on-Vaillancourt-type theorem for bilinear operators}, Indiana Univ. Math. J.  \textbf{62} (2013) 1165-1201.


\bibitem{Mu_Sc}
C. Muscalu and W. Schlag, \emph{Classical and Multilinear Harmonic Analysis}, vol. II. Cambridge University Press. Cambridge  (2013).

\bibitem{Na}
V. Naibo, \emph{On the bilinear H\"ormander classes in the scales of Triebel-Lizorkin and Besov spaces}, J. Fourier Anal. Appl.  \textbf{21} (2015) 1077-1104.

\bibitem{Na_Tho}
V. Naibo and A. Thomson, \emph{Coifman-Meyer multipliers: Leibniz-type rules and applications to scattering of solutions to PDEs}, Trans. Amer. Math. Soc. \textbf{372} (2019) 5453-5481.



\bibitem{Park1}
B. Park,  \emph{Some maximal inequalities on Triebel-Lizorkin spaces for $p=\infty$},  Math. Nachr. \textbf{292} (2019) 1137-1150.

\bibitem{Park5}
B. Park,  \emph{Fourier multipliers on a vector-valued function space}, Submitted.


\bibitem{Ro_St}
S. Rodr\'iguez-L\'opez and W. Staubach,  \emph{Estimates for rough Fourier integral and pseudodifferential operators and applications to the boundedness of multilinear operators}, J. Func. Anal. \textbf{264} (2013) 2356-2385.



\bibitem{Tom}
N. Tomita,  \emph{A H\"ormander type multiplier theorem for multilinear operators}, J. Func. Anal.  \textbf{259}
(2010) 2028-2044.



\bibitem{Tr}
H. Triebel, \emph{Theory of Function Spaces}, Birkhauser, Basel-Boston-Stuttgart
(1983).



\end{thebibliography}
\end{document}